\documentclass[
reqno]{amsart} 
\usepackage[T1]{fontenc}
\usepackage[english]{babel}
\usepackage{amssymb,bbm,enumerate}
\usepackage{color}
\usepackage[linktocpage=true,colorlinks=true, linkcolor=blue, citecolor=red, urlcolor=green]{hyperref}

\usepackage{xypic}

\usepackage{tikz-cd}

\renewcommand{\phi}{\varphi}

\def\ch{\mathrm{ch}}
\def\cl{\mathrm{cl}}

\def\si{\sigma}

\def\Ga{\Gamma}
\def\la{\lambda}
\def\La{\Lambda}

\newcommand{\mc}[1]{\mathcal{#1}}

\newcommand{\mb}[1]{\mathbb{#1}}

\newcommand{\id}{\mathbbm{1}}

\DeclareMathOperator{\Hom}{Hom}
\DeclareMathOperator{\End}{End}

\DeclareMathOperator{\ad}{ad}

\DeclareMathOperator{\Der}{Der}

\DeclareMathOperator{\sign}{sign}

\DeclareMathOperator{\gr}{gr}
\DeclareMathOperator{\dr}{dr}

\newcommand{\PV}{\mathop{\rm PV }}

\newcommand{\Har}{\mathop{\rm Har }}

\makeatletter
\def\smallunderbrace#1{\mathop{\vtop{\m@th\ialign{##\crcr
   $\hfil\displaystyle{#1}\hfil$\crcr
   \noalign{\kern3\p@\nointerlineskip}%
   \tiny\upbracefill\crcr\noalign{\kern3\p@}}}}\limits}

\theoremstyle{plain}
\newtheorem{theorem}{Theorem}[section]
\newtheorem{lemma}[theorem]{Lemma}
\newtheorem{proposition}[theorem]{Proposition}
\newtheorem{corollary}[theorem]{Corollary}

\theoremstyle{definition}
\newtheorem{definition}[theorem]{Definition}
\newtheorem{example}[theorem]{Example}

\theoremstyle{remark}
\newtheorem{remark}[theorem]{Remark}

\numberwithin{equation}{section}

\definecolor{light}{gray}{.9}

\begin{document}

\title{Poisson vertex algebra cohomology and differential Harrison cohomology}

\author{Bojko Bakalov}
\address{Department of Mathematics, North Carolina State University,
Raleigh, NC 27695, USA}
\email{bojko\_bakalov@ncsu.edu}

\author{Alberto De Sole}
\address{Dipartimento di Matematica, Sapienza Universit\`a di Roma,
P.le Aldo Moro 2, 00185 Rome, Italy}
\email{desole@mat.uniroma1.it}
\urladdr{www1.mat.uniroma1.it/$\sim$desole}

\author{Victor G. Kac}
\address{Department of Mathematics, MIT,
77 Massachusetts Ave., Cambridge, MA 02139, USA}
\email{kac@math.mit.edu}

\author{Veronica Vignoli}
\address{Dipartimento di Matematica, Sapienza Universit\`a di Roma,
P.le Aldo Moro 2, 00185 Rome, Italy}
\email{vignoli@mat.uniroma1.it}



\dedicatory{
To Nikolai Reshetikhin on his 60-th birthday.
}

\begin{abstract}
We construct a canonical map from the Poisson vertex algebra cohomology complex
to the differential Harrison cohomology complex,
which restricts to an isomorphism on the top degree.
This is an important step in the computation of Poisson vertex algebra and vertex algebra cohomologies.
\end{abstract}

\maketitle


\pagestyle{plain}

%

\medskip

\section{Introduction}\label{Intro}

The present paper is a next step in the development of the cohomology theory of vertex algebras
started in \cite{BDSHK18,BDSHK19}.
Recall (see e.g.\ \cite{BDSHK18}) that, to any linear symmetric (super)operad $\mc P$
over a field $\mb F$, one canonically associates a $\mb Z$-graded Lie superalgebra
\begin{equation}\label{eq:intro1}
W(\mc P)=\bigoplus_{k=-1}^\infty W^k(\mc P)
\,,\quad\text{ where }\quad
W^k(\mc P)=\mc P(k+1)^{S_{k+1}}
\,.
\end{equation}
An odd element $X\in W^1(\mc P)$,
satisfying $[X,X]=0$,
defines a cohomology complex $(W(\mc P),\ad X)$,
which is a differential graded Lie superalgebra.

A cohomology theory of vertex algebras is constructed by considering the operad
$\mc P_{\ch}(V)$, attached to a vector superspace $V$
with an even endomorphism $\partial$.
In order to describe this construction,
let, for $n\in\mb Z_{\geq0}$,
$$
V_n=V[\lambda_1,\dots,\lambda_n]\big/\langle\partial+\lambda_1+\dots+\lambda_n\rangle\,,
$$
where the indeterminates $\lambda_i$ have even parity and $\langle\Phi\rangle$
stands for the image of the endomorphism $\Phi$,
and let 
$$
\mc O^{\star,T}_n=\mb F[z_i-z_j,(z_i-z_j)^{-1}]_{1\leq i<j\leq n} \,.
$$
The superspace $\mc P_{\ch}(V)(n)$
is defined as the set of all linear maps
\begin{equation}\label{eq:1.13}
Y\colon V^{\otimes n}\otimes\mc O^{\star,T}_n\to V_n
\,\,,\qquad
v_1\otimes\dots\otimes v_n\otimes f\mapsto
Y_{\lambda_1,\dots,\lambda_n}(v_1\otimes\dots\otimes v_n\otimes f)
\,,
\end{equation}
satisfying the following two sesquilinearity properties ($1\leq i\leq n$):
\begin{equation}\label{eq:1.14}
Y_{\lambda_1,\dots,\lambda_n}(v_1\otimes\dots\otimes (\partial+\lambda_i)v_i \otimes\dots\otimes v_n\otimes f)
=
Y_{\lambda_1,\dots,\lambda_n}\Bigl(v_1\otimes\dots\otimes v_n\otimes \frac{\partial f}{\partial z_i}\Bigr)
\,,
\end{equation}
and
\begin{equation}\label{eq:1.15}
Y_{\lambda_1,\dots,\lambda_n}(v_1\otimes\dots\otimes v_n\otimes (z_i-z_j)f)
=
\Bigl(
\frac{\partial}{\partial\lambda_j}
-
\frac{\partial}{\partial\lambda_i}
\Bigr)
Y_{\lambda_1,\dots,\lambda_n}(v_1\otimes\dots\otimes v_n\otimes f)
\,.
\end{equation}
In \cite{BDSHK18} we also defined the action of $S_n$ on $\mc P_{\ch}(V)(n)$
and the $\circ_i$-products, making $\mc P_{\ch}(V)$ an operad.

As a result,
we obtain the Lie superalgebra 
$$
W_{\ch}(V)=W(\mc P_{\ch}(V))=\bigoplus_{k=-1}^\infty W_{\ch}^k(V)
\,,
$$
see \eqref{eq:intro1}.
We show in \cite{BDSHK18}
that odd elements $X\in W_{\ch}^1(\Pi V)$,
such that $[X,X]=0$,
correspond bijectively to vertex algebra structures on the $\mb F[\partial]$-module $V$,
such that $\partial$ is the translation operator (where $\Pi$ stands for the reversal of parity).
This leads to the vertex algebra cohomology complex $(W_{\ch}(\Pi V),\ad X)$,
with coefficients in the adjoint $V$-module.
The cohomology with coefficients in an arbitrary $V$-module $M$
in obtained by a simple reduction procedure.

Now, suppose that the $\mb F[\partial]$-module $V$
has an increasing $\mb Z_{\geq0}$-filtration by $\mb F[\partial]$-submodules.
Taking the increasing filtration of $\mc O^{\star,T}_n$ by the number of divisors,
we obtain an increasing filtration of $V^{\otimes n}\otimes\mc O^{\star,T}_n$.
This filtration induces a decreasing filtration of the superspace $\mc P_{\ch}(V)(n)$
(see \cite{BDSHK18}).
The associated graded spaces $\gr\mc P_{\ch}(V)(n)$
form a graded operad.

On the other hand, in \cite{BDSHK18}
we introduced the closely related operad $\mc P_{\cl}(V)$,
which ``governs'' the Poisson vertex algebra (PVA) structures on the $\mb F[\partial]$-module $V$.
Let $\mb F \mc G(n)$ be the vector space
(with even parity)  spanned by the set $\mc G(n)$ of labeled oriented graphs with $n$ vertices.
The vector superspace $\mc P_{\cl}(V)(n)$
is the space of linear maps (cf.\ \eqref{eq:1.13})
\begin{equation}\label{eq:1.16}
Y\colon \mb F \mc G(n) \otimes V^{\otimes n}
\to V_n \,,
\qquad \Gamma\otimes v \mapsto Y^\Gamma(v)
\,,
\end{equation}
satisfying the sesquilinearity conditions \eqref{eq:sesq1} and \eqref{eq:sesq2} in Section \ref{ssec:cloperad}, 
which are the ``classical'' analogs of \eqref{eq:1.14} and \eqref{eq:1.15}.
The corresponding $\mb Z$-graded Lie superalgebra 
$$
W_{\cl}(\Pi V)=\bigoplus_{k=-1}^\infty W_{\cl}^k(\Pi V)
\,,
$$
is such that odd elements $X\in W_{\cl}^1(\Pi V)$ with $[X,X]=0$
parametrize the PVA structures on the $\mb F[\partial]$-module $V$ by 
\begin{equation}\label{eq:1.17}
ab=(-1)^{p(a)}X^{\bullet\to\bullet}(a\otimes b) 
\,,\qquad
[a_\lambda b]=
(-1)^{p(a)}X^{\bullet\,\,\,\bullet}_{\lambda,-\lambda-\partial}(a\otimes b)
\,.
\end{equation}

When $V$ is endowed with an increasing $\mb Z_{\geq0}$-filtration by $\mb F[\partial]$-submodules,
we have a canonical linear map of graded operads
\begin{equation}\label{eq:1.18}
\gr\mc P_{\ch}(V)
\to
\mc P_{\cl}(\gr V)
\,.
\end{equation}
It is proved in \cite{BDSHK18}
that the map \eqref{eq:1.18} is injective.
The main result of \cite{BDSHK19} is that this map is an isomorphism, 
provided that the filtration of $V$
is induced by a grading by $\mb F[\partial]$-modules.
If, in addition, this filtration of $V$ is such that $\gr V$ inherits from the vertex algebra structure of $V$ a PVA structure, 
then, as a result, the vertex algebra cohomology is majorized by the classical PVA cohomology:
\begin{equation}\label{eq:1.19}
\dim H_{\ch}^n(V)
\leq
\dim H_{\cl}^n(V)
\,.
\end{equation}

Unfortunately, \eqref{eq:1.19} is not a ``practical'' inequality,
since the direct computation of $H_{\cl}(V,M)$
may be very hard.
However, 
if the superspace $V$ 
is endowed with the structure of a commutative associative superalgebra
with an even derivation $\partial$,
there exists a much smaller $\mb Z$-graded Lie superalgebra,
constructed in \cite{DSK13}:
$$
W_{\PV}(\Pi V)=\bigoplus_{k=-1}^\infty W_{\PV}^k(\Pi V)
\,.
$$
It has the property that PVA structures on the differential algebra $V$
correspond bijectively to odd elements $X\in W^1_{\PV}(\Pi V)$
such that $[X,X]=0$, cf.\  \eqref{eq:1.17}.
This produces the variational Poisson cohomology complex
$(W_{\PV}(\Pi V),\ad X)$.
It is easy to see that this complex is a subcomplex of the complex
$(W_{\cl}(\Pi V),\ad X)$,
corresponding to the graphs without edges.

The present paper is the first paper towards proving that the inclusion of complexes
\begin{equation}\label{eq:introa}
(W_{\PV}(\Pi V),\ad X)\hookrightarrow (W_{\cl}(\Pi V),\ad X)
\,,
\end{equation}
induces an isomorphism in cohomology, under some assumptions on the differential algebra $V$.
In this case, we may replace $H_{\cl}$ by $H_{\PV}$ in \eqref{eq:1.19}.
Since there are by now well developed tools for computing variational Poisson cohomology,
see \cite{DSK13} and \cite{BDSK19},
this allows one to get a hold of the vertex algebra cohomology.

In the present paper, given a vector superspace $V$ with an even endomorphism $\partial$,
we construct a canonical map of complexes
\begin{equation}\label{eq:intro9}
(W_{\cl}(\Pi V),\ad X)
\,\stackrel{\varphi}{\longrightarrow}\,
(C_{\partial,\Har}(V),d)
\,,
\end{equation}
which restricts to a bijective linear map on the top degree:
$$
\gr^{n-1}W^{n-1}_{\cl}(\Pi V)
\,\stackrel{\sim}{\longrightarrow}\,
C^n_{\partial,\Har}(V)
\,,
$$
see Theorem \ref{thm:main} in Section \ref{sec:main}.
Here $X\in W^1_{\cl}(\Pi V)$ corresponds to the PVA structure on $V$,
given by \eqref{eq:1.17}.
In particular, $V$ is endowed with a structure of a commutative associative superalgebra 
with an even derivation $\partial$,
hence we may consider the differential Harrison complex 
$$
\Big(
C_{\partial,\Har}(V)=\bigoplus_{n\in\mb Z_{\geq0}}C^n_{\partial,\Har}(V)
\,,\,
d
\Big)
\,.
$$
Here $C^n_{\partial,\Har}(V)$ is the subspace of $\Hom_{\mb F[\partial]}(V^{\otimes n},V)$
satisfying Harrison's conditions \eqref{eq:harrisoncond} in Section \ref{ssec:Harrison}.
It was shown in \cite{Har62,GS87} 
that the subspace $C_{\partial,\Har}(V)$
of $\bigoplus_{n\in\mb Z_{\geq0}}\Hom_{\mb F[\partial]}(V^{\otimes n},V)$
is invariant with respect to the Hochschild differential $d$
(which commutes with $\partial$).
Finally, the map $\varphi$ in \eqref{eq:intro9}
maps $Y\in W^{n-1}_{\cl}(\Pi V)$ to $Y^{\Lambda_n}$,
where $\Lambda_n$ is the labeled oriented graph 
$$
\begin{tikzpicture}[scale=1.5]
\node at (-0.5,0){$\La_n=$};
\draw[fill] (0,0) circle [radius=0.02];
\node at (0,-0.2){\tiny{$1$}};
\draw[->] (0,0) to (0.3,0);
\draw[fill] (0.3,0) circle [radius=0.02];
\node at (0.3,-0.2){\tiny{$2$}};
\draw[->] (0.3,0) to (0.6,0);
\node at (0.7,0) {.};
\node at (0.75,0) {.};
\node at (0.8,0) {.};
\draw[->] (0.9,0) to (1.2,0);
\draw[fill] (1.2,0) circle [radius=0.02];
\node at (1.2,-0.2){\tiny{$n$}};
\end{tikzpicture}
$$

In our next paper \cite{BDSHKV19}
we will show that if the differential superalgebra $V$
is such that $H_{\partial,\Har}^n(V,d)=0$ for $n>1$,
then the inclusion \eqref{eq:intro9}
is a quasi-isomorphism,
hence in this case the inequality \eqref{eq:1.19} turns into the inequality
\begin{equation}\label{eq:1.19a}
\dim H_{\ch}^n(V)
\leq
\dim H_{\PV}^n(V)
\,.
\end{equation}
We refer to the Ph.D.\ thesis \cite{Vig19} for examples and more details.

Throughout the paper, the base field $\mb{F}$ has characteristic 0, and, unless otherwise specified, all vector spaces, their tensor products and $\Hom$'s are over $\mb{F}$.

\subsubsection*{Acknowledgments} 

This research was partially conducted during the authors' visits
to the University of Rome La Sapienza, to MIT, and to IHES.
The first author was supported in part by a Simons Foundation grant 584741.
The second author was partially supported by the national PRIN fund n. 2015ZWST2C$\_$001
and the University funds n. RM116154CB35DFD3 and RM11715C7FB74D63.
The third author was partially supported by 
the Bert and Ann Kostant fund and by a Simons Fellowship.

\section{Differential Harrison cohomology complex} \label{sec:Harrison}

In this section, we recall the definition of the Harrison cohomology complex
and we introduce the differential Harrison cohomology complex.

\subsection{Hochschild cohomology complex} \label{ssec:Hoch}

First, we review the Hochschild cohomology complex, 
of which Harrison's is a subcomplex, see \cite{Hoc45} and \cite{Har62}.
We use the original Harrison's definition. For other definitions see \cite{GS87,L13}. 

Let $A$ be an associative algebra over the base field $\mb F$,
and $M$ be an $A$-bimodule. We will write $A^{\otimes n}$ 
for the $n$-fold tensor product $A\otimes \dots \otimes A$.
The Hochschild cohomology complex is defined as follows.
The space of $n$-cochains is
\begin{equation} \label{eq:hochobject}
\Hom(A^{\otimes n},M)\;,
\end{equation}
and the differential $d\colon\Hom(A^{\otimes n},M) \rightarrow \Hom(A^{\otimes n+1},M)$ is defined by
\begin{align} \label{eq:hochdifferential}
(df)(a_1\otimes &\dots\otimes a_{n+1}) = a_1f(a_2\otimes\dots\otimes a_{n+1}) \notag \\ &+\sum_{i=1}^{n} (-1)^i f(a_1\otimes \dots\otimes a_{i-1}\otimes a_ia_{i+1}\otimes a_{i+2}\otimes \dots\otimes a_{n+1}) \notag \\
&+ (-1)^{n+1} f(a_1\otimes\dots\otimes a_n)a_{n+1}\,.
\end{align}
Then $d^2=0$, and we get the Hochschild cohomology complex
\begin{equation}\label{eq:hochcomplex}
0 \longrightarrow M \stackrel{d}{\longrightarrow} \Hom(A,M)\stackrel{d}{\longrightarrow} \Hom(A^{\otimes2},M) \stackrel{d}{\longrightarrow}\cdots \,.
\end{equation}

If $A$ is an associative algebra with a derivation $\partial\colon A\rightarrow A$, 
and $M$ is a differential bimodule over $A$
(i.e., the action of $\partial$ is compatible with the bimodule structure), 
we may consider the \emph{differential Hochschild cohomology complex} 
by taking the subspace of $n$-cochains
\begin{equation}\label{eq:cochaindiff}
	\Hom_{\mb F[\partial]}(A^{\otimes n}, M)\,.
\end{equation}
It is clear by the definition \eqref{eq:hochdifferential} that the differential $d$ 
maps $\Hom_{\mb F[\partial]}(A^{\otimes n}, M)$ to $\Hom_{\mb F[\partial]}(A^{\otimes n+1}, M)$. 
Hence, we have a cohomology subcomplex.

\begin{remark}\label{rem:super}
It is straightforward, using the Koszul--Quillen rule,
to extend the definition of the Hochschild complex
to the case when $A$ is an associative superalgebra,
as well as all other definitions and results of the paper.
We restricted here to the purely even case for the simplicity of the exposition.
\end{remark}

\subsection{Monotone permutations} \label{ssec:monotone}

Consider the symmetric group $S_n$. Using Harrison's notation in \cite{Har62} (see also \cite{GS87}), 
we have the following definition:
\begin{definition}
A permutation $\pi \in S_n$ is called \emph{monotone} if, for each $i=1,\dots , n$, 
one of the following two conditions holds:
\begin{enumerate}[(a)]
\item $\pi(j)<\pi(i)$ for all $j<i$;
\item$\pi(j)>\pi(i)$ for all $j<i$.
\end{enumerate}
(Not necessarily the same condition (a) or (b) holds for every $i$.) When (b) holds, 
we call $i$ a \emph{drop} of $\pi$. Also, $\pi(1)=k$ is called the \emph{start} of $\pi$
(and we say that $\pi$ \emph{starts} at $k$).
\end{definition}

We denote by $\mc M_n\subset S_n$ the set of monotone permutations, and by $\mc M_n^k\subset \mc M_n$ 
the set of monotone permutations starting at $k$.

Here is a simple description of all monotone permutations starting at $k$. 
Let us identify the permutation $\pi\in S_n$ with the $n$-tuple $[\pi(1),\dots,\pi(n)]$. 
To construct all $\pi\in\mc M_n^k$, we let $\pi(1)=k$.
Then, for every choice of $k-1$ positions in $\{2,\dots,n\}$ we get a monotone permutation $\pi$ as follows. 
In the selected positions we put the numbers $1$ to $k-1$ 
in decreasing order from left to right; in the remaining positions we write the numbers $k+1$ to $n$ 
in increasing order from left to right.
(The selected positions are the drops of $\pi$.)

According to the above description, we have a bijective correspondence
\begin{equation}\label{eq:correspondence}
\mc M_n^k \xrightarrow[]{\sim} 
\big\{
D\subset \{2,\dots,n\} 
\,\big|\;
|D|=k-1 
\big\}\,,
\end{equation}
associating the monotone permutation $\pi\in \mc M_n^{k}$ 
to the set $D(\pi)$ of drops of $\pi$, which are 
$$
\pi^{-1}(k-1)<\pi^{-1}(k-2)<\dots <\pi^{-1}(1)\in \{2,\dots,n\}
\,.
$$

\begin{example}
The only monotone permutation starting at $1$ is the identity, while the only monotone permutation starting at $n$ is 
\begin{equation}\label{eq:sigman}
\sigma_n=[n\;\;n-1\;\cdots\;2\;\;1]\,.
\end{equation}
\end{example}
\begin{example}
Let $n=5$ and $k=3$. The monotone permutations starting at $3$ are 
\begin{align*}
& [3\;\;\underline{2}\;\;\underline{1}\;\;4\;\;5]\,,\quad
[3\;\;\underline{2}\;\;4\;\;\underline{1}\;\;5]\,,\quad
[3\;\;\underline{2}\;\;4\;\;5\;\;\underline{1}] \,,\\
& [3\;\;4\;\;\underline{2}\;\;\underline{1}\;\;5] \,,\quad
[3\;\;4\;\;\underline{2}\;\;5\;\;\underline{1}] \,,\quad
[3\;\;4\;\;5\;\;\underline{2}\;\;\underline{1}] \,,
\end{align*}
where we underlined the positions of the drops.
\end{example}
Given a monotone permutation $\pi$, we denote by $\dr(\pi)$ the sum of all the drops with respect to $\pi$. 
According to the previous description, we can easily see that 
\begin{equation} \label{eq:drop}
(-1)^{\dr(\pi)}=(-1)^{k-1} \sign (\pi)\,,
\end{equation} 
if $k$ is the start of $\pi$.

Note that the description \eqref{eq:correspondence} of $\mc M_n^{k}$ in terms of positions of drops allows us to count the number of elements in $\mc M_n^k$, for fixed $n$ and $k$. We have: 
\begin{equation}\label{eq:A}
|\mc M_n^k| = \binom{n-1}{k-1} \,.
\end{equation}
\begin{remark}\label{rem:monotone}
Let us denote by $\mc M_n^{k,k-1}\subset \mc M_n^k$ the subset of all monotone permutations $\pi$ 
starting at $k$ with $\pi(2)=k-1$, and by $\mc M_n^{k,k+1}\subset \mc M_n^k$ the subset 
of all monotone permutations $\pi$ starting at $k$ with $\pi(2)=k+1$. Then we have
\begin{equation*}
\mc M_n^k =	\mc M_n^{k,k-1} \sqcup \mc M_n^{k,k+1}\,.
\end{equation*}
\end{remark}
\begin{lemma} \label{lem:1}
There are natural identifications 
$$
\mc M_n^{k,k-1} \simeq \mc M_{n-1}^{k-1}
\quad
(\text{resp. }
\mc M_n^{k,k+1} \simeq \mc M_{n-1}^{k})
\,,
$$
mapping $\pi\in \mc M_n^{k,k-1}$ 
to $\bar{\pi}\in\mc M_{n-1}^{k-1}$ $($resp. $\pi\in \mc M_n^{k,k+1}$ to $\bar{\pi}\in\mc M_{n-1}^{k})$, 
given by $\bar{\pi}(1):= k-1$ $($resp. $k)$, and, for $i=2,\dots, n-1$,
\begin{equation}
\bar{\pi}(i):= 
\begin{cases}
\pi(i+1)\,, & \;\text{if \; $\pi(i+1)<k$} \\
\pi(i+1)-1\,,  & \;\text{if \; $\pi(i+1)>k$}
\end{cases}\;.
\end{equation}
Moreover, $$(-1)^{\dr(\bar{\pi})}=(-1)^{\dr(\pi)+k}\, \quad (\text{resp. } (-1)^{\dr(\pi)+ k-1}) \,.$$ 
\end{lemma}
\begin{proof}
Straightforward; see \cite{Vig19} for details.
\end{proof}
\begin{remark} \label{rem:monotone2}
Observe that, given a monotone permutation $\pi$, either $\pi(n)=1$ or $\pi(n)=n$. 
Denote by ${}^{1}{\mc{M}}_n^k\subset \mc{M}_n^k$ the set of all the monotone permutations $\pi$
starting at $k$ with $\pi(n)=1$, and by ${}^{n}{\mc{M}}_n^k\subset \mc{M}_n^k$ the set of 
all the monotone permutations $\pi$ starting at $k$ with $\pi(n)=n$. 
As in Remark \ref{rem:monotone}, we have
\begin{equation*}
\mc M_n^k =	{}^{1}{\mc M}_n^{k} \sqcup {}^{n}{\mc M}_n^{k}\,.
\end{equation*}
\end{remark}
\begin{lemma} \label{lem:2}
There are natural identifications 
$$
{}^{1}{\mc M}_n^{k}\simeq \mc{M}_{n-1}^{k-1}
\quad
(\text{resp. }
{}^{n}{\mc M}_n^{k}\simeq \mc{M}_{n-1}^{k})
\,,
$$
mapping $\pi\in {}^{1}{\mc{M}}_n^k$ to $\tilde{\pi}\in \mc M_{n-1}^{k-1}$  
$($resp. $\pi\in {}^{n}{\mc{M}}_n^k$ to $\tilde{\pi}\in \mc M_{n-1}^{k})$, 
given by $\tilde{\pi}(i):= \pi(i)-1$ $($resp. $\tilde{\pi}(i):= \pi(i))$, for $i=1,\dots,n-1$. 
Moreover, 
$$
(-1)^{\dr(\tilde{\pi})}=(-1)^{\dr(\pi)+n}\, \quad (\text{resp. } (-1)^{\dr(\pi)})
\,.
$$ 
\end{lemma}
\begin{proof}
Straightforward; see \cite{Vig19} for details.
\end{proof}

\subsection{Differential Harrison cohomology complex} \label{ssec:Harrison}

Let us now recall Harrison's original definition of his cohomology complex \cite{Har62}. 
Let $A$ be a commutative associative algebra, and $M$ be a symmetric $A$-bimodule, 
i.e., such that $am=ma$, for all $a\in A$ and $m\in M$. 
For every $1<k\leq n$
define the following endomorphism on the space $\Hom(A^{\otimes n},M)$:
\begin{equation} \label{eq:lkf}
(L_kF)(a_1\otimes\dots\otimes a_n):= \sum_{\pi \in \mc M_n^k} (-1)^{\dr(\pi)} F(a_{\pi(1)}\otimes\dots\otimes a_{\pi(n)})\,.
\end{equation}
A \emph{Harrison $n$-cochain} is defined as a Hochschild $n$-cochain
$F\in\Hom(A^{\otimes n},M)$ fixed by all operators $L_k$:
\begin{equation} \label{eq:harrisoncond}
L_kF=F\,, \;\text{ for every } \; 2\leq k\leq n\,.
\end{equation}
We will denote by 
\begin{equation} \label{eq:harcochain}
C^n_{\Har}(A,M)\subset \Hom(A^{\otimes n},M)
\end{equation}
the space of Harrison $n$-cochains.

Furthermore, if $A$ is a differential algebra with a derivation $\partial:A\rightarrow A$, 
and $M$ is a symmetric differential bimodule, we may consider the space of 
\emph{differential Harrison} $n$-cochains
\begin{equation}\label{eq:harrisondiff}
C^{n}_{\partial,\Har}(A,M)\subset \Hom_{\mb{F}[\partial]}(A^{\otimes n},M)\,,
\end{equation}
again defined by Harrison's conditions \eqref{eq:harrisoncond}.
\begin{proposition}
\begin{enumerate}[(a)]
\item
The Harrison complex $(C_{\Har}(A,M),d)$ is a subcomplex of the Hochschild complex.
\item
If $A$ is a differential algebra, with a derivation $\partial:A\rightarrow A$, 
the differential Harrison complex $(C_{\partial,\Har}(A,M),d)$ is a subcomplex 
of the differential Hochschild complex.
\end{enumerate}
\end{proposition}
\begin{proof}
The proof of (a) is in \cite{Har62,GS87}.
Part (b) is straightforward.
\end{proof}

\begin{remark}\label{rem:har}
Clearly, $H^0_{\partial,\Har}(A,M)=M$
and $H^1_{\partial,\Har}(A,M)=\Der_{\mb F[\partial]}(A,M)$.
It follows from \cite{GS87} that $H^n_{\partial,\Har}(A,M)$
is a direct summand of the differential Hochschild cohomology 
$HH^n_{\partial}(A,M)$, for $n\geq2$.
\end{remark}

\section{The classical operad and PVA cohomology} \label{sec:PVAcohom}

In this section, we recall some basic notions that will be used throughout the paper, 
and review the construction of the PVA cohomology complex as described in \cite{BDSHK18}. 

\subsection{Symmetric group actions} \label{ssec:symact}

There is a natural left action of $S_n$ on an arbitrary $n$-tuple of objects $(\lambda_{1},\dots , \lambda_n)$:
\begin{equation}\label{eq:acttuple}
\sigma (\lambda_{1},\dots , \lambda_n) = (\lambda_{\sigma^{-1}(1)},\dots , \lambda_{\sigma^{-1}(n)})
\,, \;\;\;\; \sigma\in S_n\,.
\end{equation}
Also, given $V=V_{\bar 0}\oplus V_{\bar 1}$ a vector superspace with parity $p$, we have a linear left action 
of the symmetric group $S_n$ on the tensor product $V^{\otimes n}$ $(\sigma\in S_n,\, v_1,\dots,v_n\in V)$:
\begin{equation}\label{eq:acttensorprod}
\sigma(v_1\otimes\dots\otimes v_n) := \epsilon_v(\sigma) 
\,
v_{\sigma^{-1}(1)}\otimes\dots\otimes v_{\sigma^{-1}(n)}
\,,
\end{equation}
where, following the Koszul--Quillen rule,
\begin{equation}\label{eq:signacttensorprod}
\epsilon_v(\sigma)
:=
\prod_{i<j \,:\, \sigma(i)>\sigma(j)}(-1)^{p(v_i)p(v_j)}
\,.
\end{equation}
In particular, 
if $V$ is purely even $\epsilon_v(\sigma)=1$,
while if $V$ is purely odd $\epsilon_v(\sigma)=\sign(\sigma)$.
The corresponding right action of $S_n$ on the the space $\Hom(V^{\otimes n},V)$ 
is given by $(f\in \Hom(V^{\otimes n}, V),\,\sigma\in S_n)$:
\begin{equation} \label{eq:acthom}
f^{\sigma}(v_1\otimes\dots\otimes v_n)=f(\sigma(v_1\otimes\dots\otimes v_n))\;.
\end{equation}

\subsection{Composition of permutations and shuffles}

Let $n\geq 1$ and $m_1\,,\dots\,,m_n\geq 0$. We introduce the following notation:
\begin{equation} \label{eq:notation}
M_0=0
\,\,\text{ and }\,\,
M_i=\sum_{j=1}^i m_j\,, \quad i=1,\dots,n\,.
\end{equation}
Given $\sigma\in S_n$ and $\tau_1\in S_{m_1}, \dots, \tau_n\in S_{m_n}$, we describe the \emph{composition}  
$$\sigma(\tau_1,\dots,\tau_n)\in S_{M_n}$$
by saying how it acts on the tensor power $V^{\otimes M_n}$ of a vector space $V$:
\begin{equation}\label{eq:comp}
(\sigma(\tau_1,\dots,\tau_n))(v_1\otimes \dots\otimes v_{M_n})
= \sigma(\tau_1(v_1\otimes\dots\otimes v_{M_1})\otimes\dots\otimes 
\tau_n(v_{M_{n-1}+1}\otimes\dots\otimes v_{M_n}))\,.
\end{equation}
\begin{definition}
A permutation $\sigma \in S_{m+n}$ is called an $(m,n)$-\emph{shuffle} if 
\begin{equation*}
\si(1)< \dots < \si(m)\,,\quad \si(m+1)< \dots <\si(m+n)\,.
\end{equation*}
The subset of $(m,n)$-shuffles is denoted by $S_{m,n}\subset S_{m+n}$. 
\end{definition}
Observe that, by definition, $S_{0,n}=S_{n,0}=\{1\}$ for every $n\geq 0$. If either $m$ or $n$ is negative, 
we set $S_{m,n}= \emptyset$ by convention. 

\subsection{$n$-graphs}\label{ssec:ngraphs}

For an oriented graph $\Gamma$, we denoted by $V(\Gamma)$ the set of vertices of $\Ga$, 
and by $E(\Gamma)$ the set of edges. We call $\Ga$ an $n$-\emph{graph} 
if $V(\Gamma)=\{1, \dots, n\}$. Denote by $\mc G(n)$ the set of all $n$-graphs without tadpoles, 
and by $\mc G_0(n)$ the set of all acyclic $n$-graphs. 

An $n$-graph $L$ will be called an $n$-\emph{line}, or simply a \emph{line}, 
if its set of edges is of the form $\{i_1\to i_2,\,i_2\to i_3,\dots,\,i_{n-1}\to i_n\}$,
where $\{i_1,\dots,i_n\}$ is a permutation of $\{1,\dots,n\}$.

We have a natural left action of $S_n$ on the set $\mc G(n)$: 
for the $n$-graph $\Ga$ and the permutation $\sigma$, the new $n$-graph $\sigma(\Ga)$ 
is defined to be the  same graph as $\Ga$ but with the vertex which was labeled as $i$ relabeled 
as $\sigma(i)$, for every $i=1,\dots,n$. 
So, if the $n$-graph $\Ga$ has an oriented edge $i \rightarrow j$, 
then the $n$-graph $\sigma(\Ga)$ has the oriented edge $\sigma(i) \rightarrow \sigma(j)$. 
Note that $S_n$ permutes the set of $n$-lines.

Let us recall the cocomposition of $n$-graphs, as described in \cite{BDSHK18}.
Given an $n$-tuple $(m_1,\dots,m_n)$ of positive integers, let $M_i$ be as in \eqref{eq:notation}. 
If $\Ga\in \mc G(M_n)$, define $\Delta_i^{m_1,\dots,m_n}(\Ga)\in \mc G(m_i)$, $i=1,\dots,n$, 
to be the subgraph of $\Ga$ associated to the set of vertices $\{M_{i-1}+1,\dots,M_{i}\}$, 
relabeled as $\{1,\dots,m_i\}$. Define also $\Delta_0^{m_1,\dots,m_n}(\Ga)$ to be the graph obtained 
from $\Ga$ by collapsing the vertices and the edges of each $\Delta_i^{m_1,\dots,m_n}(\Ga)$ 
into a single vertex, relabeled as $i$. Then the cocomposition map is the map
\begin{align} \label{eq:cocomp}
\Delta^{m_1,\dots,m_n}\colon \mc G(M_n) &\rightarrow \mc G(n)\times\mc G(m_1)\times \dots 
\times \mc G(m_n) \,, \\
\Ga &\mapsto \left(\Delta_0^{m_1,\dots,m_n}(\Ga),\,\Delta_1^{m_1,\dots,m_n}(\Ga),\dots,
\Delta_n^{m_1,\dots,m_n}(\Ga)\right). \notag
\end{align}	
\begin{example}
Let $n=3$, $(m_1,m_2,m_3)=(3,1,4)$, and $\Ga\in \mc G(8)$ be the following graph
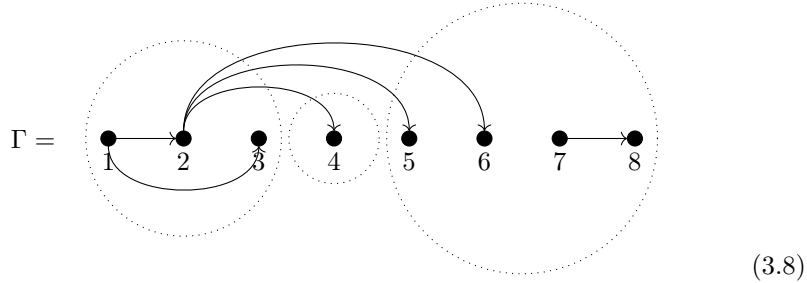
\begin{equation}\label{eq:graphexample}
\begin{tikzpicture}
\node at (-1,0) {$\Gamma=$};
\draw[fill] (0,0) circle [radius=0.1];
\node at (0,-0.3) {1};
\draw[fill] (1,0) circle [radius=0.1];
\node at (1,-0.3) {2};
\draw[fill] (2,0) circle [radius=0.1];
\node at (2,-0.3) {3};
\draw[fill] (3,0) circle [radius=0.1];
\node at (3,-0.3) {4};
\draw[fill] (4,0) circle [radius=0.1];
\node at (4,-0.3) {5};
\draw[fill] (5,0) circle [radius=0.1];
\node at (5,-0.3) {6};
\draw[fill] (6,0) circle [radius=0.1];
\node at (6,-0.3) {7};
\draw[fill] (7,0) circle [radius=0.1];
\node at (7,-0.3) {8};
\draw[->] (0.1,0) -- (0.9,0);
\draw[->] (0,-0.1) to [out=270,in=270] (2,-0.1);
\draw[->] (1,0.1) to [out=90,in=90] (3,0.1);
\draw[->] (1,0.1) to [out=90,in=90] (5,0.1);
\draw[->] (1,0.1) to [out=90,in=90] (4,0.1);
\draw[->] (6.1,0) -- (6.9,0);
\draw [dotted,thin] (1,0) circle [radius=1.3];
\draw [dotted,thin] (3,0) circle [radius=0.6];
\draw [dotted,thin] (5.5,0) circle [radius=1.8];
\end{tikzpicture}
\end{equation}
The cocomposition $\Delta^{3,1,4}(\Ga)=\bigl(\Delta_0^{3,1,4}(\Ga),\,\Delta_1^{3,1,4}(\Ga),
\Delta_2^{3,1,4}(\Ga),\Delta_3^{3,1,4}(\Ga)\bigr)$ is given by the following graphs. 
$\Delta_1^{3,1,4}(\Ga)$ is the subgraph of $\Ga$ generated by the first three vertices:
\begin{equation*}
\begin{tikzpicture}
\node at (-1.5,0) {$\Delta_1^{3,1,4}(\Ga)=$};
\draw[fill] (0,0) circle [radius=0.1];
\node at (0,-0.3) {1};
\draw[fill] (1,0) circle [radius=0.1];
\node at (1,-0.3) {2};
\draw[fill] (2,0) circle [radius=0.1];
\node at (2,-0.3) {3};
\draw[->] (0.1,0) -- (0.9,0);
\draw[->] (0,-0.1) to [out=270,in=270] (2,-0.1);
\node at (3,0) {$\in \mc G(3)\,;$};
\end{tikzpicture}
\end{equation*}
$\Delta_2^{3,1,4}(\Ga)$ is the subgraph of $\Ga$ associated to the fourth vertex:
\begin{equation*}
\begin{tikzpicture}
\node at (-1.5,0) {$\Delta_2^{3,1,4}(\Ga)=$};
\draw[fill] (0,0) circle [radius=0.1];
\node at (0,-0.3) {1};
\node at (1,0) {$\in \mc G(1)\,;$};
\end{tikzpicture}
\end{equation*}
$\Delta_3^{3,1,4}(\Ga)$ is the subgraph of $\Ga$ associated to the last four vertices:
\begin{equation*}
\begin{tikzpicture}
\node at (2.5,0) {$\Delta_3^{3,1,4}(\Ga)=$};
\draw[fill] (4,0) circle [radius=0.1];
\node at (4,-0.3) {1};
\draw[fill] (5,0) circle [radius=0.1];
\node at (5,-0.3) {2};
\draw[fill] (6,0) circle [radius=0.1];
\node at (6,-0.3) {3};
\draw[fill] (7,0) circle [radius=0.1];
\node at (7,-0.3) {4};
\draw[->] (6.1,0) -- (6.9,0);
\node at (8,0) {$\in \mc G(4)\,;$};
\end{tikzpicture}
\end{equation*}
and, finally, $\Delta_0^{3,1,4}(\Ga)$ is:
\begin{equation*}
\begin{tikzpicture}
\node at (-1.5,0) {$\Delta_0^{3,1,4}(\Ga)=$};
\draw[fill] (0,0) circle [radius=0.1];
\node at (0,-0.3) {1};
\draw[fill] (1,0) circle [radius=0.1];
\node at (1,-0.3) {2};
\draw[fill] (2,0) circle [radius=0.1];
\node at (2,-0.3) {3};
\draw[->] (0,0.1) to [out=90,in=90] (1,0.1);
\draw[->] (0,0.1) to [out=90,in=90] (2,0.1);
\draw[->] (0,-0.1) to [out=270,in=270] (2,-0.1);
\node at (4,0) {$\in \mc G(3)\,.$};
\end{tikzpicture}
\end{equation*}
\end{example}

From the construction of $\Delta_i^{m_1,\dots,m_n}(\Ga)$, it is easy to see that there is 
a natural bijective correspondence
\begin{equation}\label{eq:delta}
\Delta \colon E(\Ga) \,\stackrel{\sim}{\longrightarrow}\,
E\big(\Delta^{m_1\dots m_n}_0(\Gamma)\big) \sqcup E\big(\Delta^{m_1\dots m_n}_1(\Gamma)\big) 
\sqcup\dots\sqcup E\big(\Delta^{m_1\dots m_n}_n(\Gamma)\big)\,.
\end{equation}
\begin{definition}\label{def:extconnected}
Let $k\in\{1,\dots,M_n\}$ and $j\in\{1,\dots,n\}$.
We say that $j$ is \emph{externally connected} to $k$ (via the graph $\Gamma$
and its cocomposition $\Delta^{m_1\dots m_n}(\Gamma)$)
if there is an unoriented path (without repeating edges) of $\Delta^{m_1\dots m_n}_0(\Gamma)$
joining $j$ to $i$, where $i\in\{1,\dots,n\}$ is such that $k\in\{M_{i-1}+1,\dots,M_i\}$,
and the edge out of $i$ is the image, via the map $\Delta$ in \eqref{eq:delta},
of an edge which has its head or tail in $k$. Given a set of variables $x_1,\dots,x_{n}$,
we denote
\begin{equation}\label{eq:X(k)}
X(k) = \sum_{\substack{j \text{ externally} \\ \text{connected to } k}} x_j
\,.
\end{equation}
\end{definition}

\begin{example}
For the graph \eqref{eq:graphexample}, we have
\begin{align*}
&X(1)=0, \;X(2)=x_1+x_2+x_3, \;X(3)=0, \;X(4)=x_1+x_3,\; X(5)=x_1+x_2+x_3, \\
&X(6)=x_1+x_2+x_3, \;X(7)=0, \;X(8)=0.
\end{align*}
\end{example}

\subsection{Lie conformal algebras and Poisson vertex algebras} \label{ssec:lieconformalalgPVA}

\begin{definition}
A Lie conformal (super)algebra is a vector (super)space $V$, 
endowed with an even endomorphism $\partial\in\End(V)$
and a bilinear (over $\mb F$) $\lambda$-bracket
$[\cdot\,_\lambda\,\cdot]\colon V\times V\to V[\lambda]$
satisfying 
sesquilinearity ($a,b\in V$):
\begin{equation}\label{eq:sesquilinearity}
[\partial a_\lambda b]=-\lambda[a_\lambda b]
\,,\qquad
[a_\lambda \partial b]=(\lambda+\partial)[a_\lambda b]
\,,
\end{equation}
skewsymmetry ($a,b\in V$):
\begin{equation}\label{eq:skewsymmetry}
[a_\lambda b]=-(-1)^{p(a)p(b)}[b_{-\lambda-\partial}a]
\,,
\end{equation}
and the Jacobi identity ($a,b,c\in V$):
\begin{equation}\label{eq:jacobi}
[a_{\lambda}[b_\mu c]]-(-1)^{p(a)p(b)}[b_\mu [a_\lambda ,b]]
=[[a_\lambda b]_{\lambda+\mu}c]
\,.
\end{equation}
\end{definition}
\begin{definition}	
A Poisson vertex (super)algebra (PVA) is a commutative associative (super)algebra $V$ 
endowed with an even derivation $\partial$ 
and a Lie conformal (super)algebra $\lambda$-bracket $[\cdot_{\lambda} \cdot]$ 
that satisfies the left Leibniz rule
\begin{equation} \label{eq:leibniz}
[a_{\la} bc] = [a_\la b]c + (-1)^{p(a)p(b)}b[a_\la c]\,.
\end{equation}
\end{definition}

\subsection{Operads} \label{ssec:operad}

Recall that 
a (linear, unital, symmetric) \emph{superoperad} $\mc P$ is a collection of vector superspaces 
$\mc P(n)$, $n\geq0$, with parity $p$,
endowed, for every $f\in\mc P(n)$ and $m_1,\dots,m_n\geq0$, with a \emph{composition} 
parity preserving linear map, 
\begin{equation}\label{eq:operad1}
\begin{split}
\mc P(n) \otimes \mc P(m_1)\otimes\dots\otimes\mc P(m_n)\,\,&\to\,\,\mc P(M_n) \,, \\
f \otimes g_1 \otimes\dots\otimes g_n \,\, &\mapsto \,\, f(g_1\otimes\dots\otimes g_n) \,,
\end{split}
\end{equation}
where $M_n$ is as in \eqref{eq:notation},
satisfying the following associativity axiom:
\begin{equation}\label{eq:operad2}
f\big(
(g_1\otimes\dots\otimes g_n)
(h_1\otimes\dots\otimes h_{M_n})
\big)
=
\big(f
(g_1\!\otimes\dots\otimes\! g_n)
\big)
(h_1\otimes\dots\otimes h_{M_n})
\,\in\mc P\Big(\sum_{j=1}^{M_n}\ell_j\Big)
\,,
\end{equation}
for every $f\in\mc P(n)$, $g_i\in\mc P(m_i)$ for $i=1,\dots,n$,
and $h_{j}\in\mc P(\ell_{j})$ for $j=1,\dots,M_n$.
In the left-hand side of \eqref{eq:operad2}
the linear map 
$$g_1\!\otimes\dots\otimes\! g_n\colon \bigotimes_{j=1}^{M_n}\mc P(\ell_j)
\to\bigotimes_{i=1}^n\mc P\Bigl(\sum_{j=M_{i-1}+1}^{M_i}\ell_{j}\Bigr)$$
is the tensor product of composition maps applied to
$$
h_1\otimes\dots\otimes h_{M_n} = (h_1\otimes\dots\otimes h_{M_1}) 
\otimes (h_{M_1+1}\otimes\dots\otimes h_{M_2})
\otimes\dots\otimes (h_{M_{n-1}+1}\otimes\dots\otimes h_{M_n}).
$$

We assume that $\mc P$ is endowed with a \emph{unit} element $1\in\mc P(1)$
satisfying the following unity axioms: 
\begin{equation}\label{eq:operad3}
f(1\otimes\dots\otimes 1)=1(f)=f
\,,\,\,\text{ for every }\,\,f\in\mc P(n)
\,.
\end{equation}
Furthermore, we assume that,
for each $n\geq1$, $\mc P(n)$ has a right action of the symmetric group $S_n$,
denoted $f^\sigma$, for $f\in\mc P(n)$ and $\sigma\in S_n$,
satisfying the following equivariance axiom
($f\in\mc P(n)$, $g_1\in\mc P(m_1),\dots,g_n\in\mc P(m_n)$,
$\sigma\in S_n$, $\tau_1\in S_{m_1},\dots,\tau_n\in S_{m_n}$):
\begin{equation}\label{eq:operad4}
f^\sigma(g_1^{\tau_1}\otimes \dots\otimes g_n^{\tau_n})
=
\big(f(\sigma(g_1\otimes\dots\otimes g_n))\big)^{\sigma(\tau_1,\dots,\tau_n)}
\,,
\end{equation}
where the left action of $\sigma\in S_n$ on the tensor product of vector superspaces
was defined in \eqref{eq:acttensorprod}, and the composition $\sigma(\tau_1,\dots,\tau_n)$ 
is described in \eqref{eq:comp}.

For simplicity, from now on, we will use the term operad in place of superoperad. Given an operad $\mc P$, 
one defines, for each $i=1,\dots,n$, 
the $\circ_i$-product
$\circ_i\colon\mc P(n)\otimes\mc P(m)\to\mc P(n+m-1)$
by insertion in position $i$, i.e.
\begin{equation}\label{eq:operad8}
f\circ_i g=f(
\overbrace{1\otimes\dots\otimes 1}^{i-1}
\otimes\stackrel{\vphantom{\Big(}i}{g}\otimes
\overbrace{1\otimes\dots\otimes1}^{n-i})\,.
\end{equation}
\begin{example}
The simplest example of an operad is $\mc P= \mc{H}om$. Given a vector superspace $V$, 
$\mc{H}om=\mc{H}om(V)$ is defined as the collection of
\begin{equation*}
\mc{H}om(n):= \Hom(V^{\otimes n}, V)\,,
\qquad n\geq 0 \,,
\end{equation*}
endowed with the composition maps ($f\in\mc{H}om(n)$, $g_i\in \mc{H}om(m_i)$ for $i=1,\dots,n$, $v_j\in V$ 
for $j=1,\dots,M_n$)
\begin{equation*}
(f(g_1\otimes\dots\otimes g_n))(v_1\otimes\dots\otimes v_{M_n}):= f((g_1\otimes\dots
\otimes g_n) (v_1\otimes\dots\otimes v_{M_n}))\,,
\end{equation*}
where $M_n$ is as in \eqref{eq:notation}. $\mc{H}om$ is a unital operad with unity $1={\id}_{V}\in \End(V)$, 
and the right action of $S_n$ on $\mc{H}om(n)$ is given by \eqref{eq:acthom}.
\end{example}

\subsection{The $\mb{Z}$-graded Lie superalgebra associated to an operad} \label{ssec:liesuperalgebra}

Recall that, given an operad $\mc P$, one can construct the associated $\mb{Z}$-graded Lie superalgebra $W (\mc P)$. It is defined, as a $\mb Z$-graded vector superspace
\begin{equation}\label{eq:liesuperalg}
W (\mc P) 	= \sum_{n\ge -1} W^n (\mc P) = \sum_{n\ge -1} \mc P(n+1)^{S_{n+1}}\,,
\end{equation}
with the following Lie bracket.
For $f\in W^n(\mc P)$ and $g\in W^m(\mc P)$, their $\Box$-product is defined by
\begin{equation}\label{eq:box}
f\Box g
=
\sum_{\sigma\in S_{m+1,n}}
(f\circ_1 g)^{\sigma^{-1}}
\,\in W^{m+n}(\mc P)\,
\end{equation}
and the Lie bracket on $W(\mc P)$ is given by 
\begin{equation}\label{eq:liebracket}
[f,g] = f\Box g -(-1)^{p(f)p(g)} g\Box f \,.
\end{equation}
See e.g.\ \cite[Sec. 3]{BDSHK18} for details.

\subsection{The classical operad $\mc P_{\cl}$ \cite{BDSHK18}}\label{ssec:cloperad}

Let $V=V_{\bar 0}\oplus V_{\bar 1}$ be a vector superspace with parity $p$, endowed with an even 
endomorphism $\partial\in\End V$. For $n\ge 0$, define $\mc P_{\cl} (n)$ as the vector superspace of all maps
\begin{equation}\label{eq:operadcl}
f\colon \mc G(n)\times V^{\otimes n} \longrightarrow 
V[\lambda_1,\dots,\lambda_n]\big/\langle\partial+\lambda_1+\dots+\lambda_n\rangle\,,
\end{equation}
which are linear in the second factor, mapping the $n$-graph $\Gamma\in\mc G(n)$ 
(by definition $p(\Gamma)=\bar0$)
and the monomial $v_1\otimes\,\cdots\,\otimes v_n\in V^{\otimes n}$ to the polynomial
\begin{equation}\label{eq:imcl}
f^{\Gamma}_{\lambda_1,\dots,\lambda_n}(v_1\otimes\,\cdots\,\otimes v_n)\,,
\end{equation}
satisfying the {cycle relations} and the {sesquilinearity conditions} described as follows.

The \emph{cycle relations} say that
\begin{equation}\label{eq:cycle1}
\text{if }\,\Gamma \notin \mc G_0(n) \,,\;\;\text{ then}\;\; f^{\Gamma}=0 \,,
\end{equation}
and 
\begin{equation}\label{eq:cycle2}
\text{if }\, C\subset E(\Gamma) \,\;\text{ is an oriented cycle of } \, \Gamma\,,\,\;\text{ then }\,\;
\sum_{e\in C}f^{\Gamma\backslash e} =0 \,,
\end{equation}
where $\Gamma\backslash e$ is the graph obtained from $\Gamma$ by removing the edge $e$.
Observe that for oriented cycles of length $2$, the cycle relation \eqref{eq:cycle2} means 
that changing orientation of a single edge of the $n$-graph $\Gamma\in\mc G(n)$ amounts 
to a change of sign of $f^{\Gamma}$.

The \emph{sesquilinearity conditions} are as follows.
Let $\Gamma=\Gamma_1\sqcup\dots\sqcup\Gamma_s$
be the decomposition of $\Gamma$ as a disjoint union of its connected components,
and let $I_1,\dots,I_s\subset\{1,\dots,n\}$ be the sets of vertices associated to these
connected components. 
For a graph $\tilde{\Ga}$ and its set of vertices $\tilde{I}\subset \{1,\dots,n\}$, we write
\begin{equation}\label{eq:notation3}
\lambda_{\tilde{\Ga}} 
= \sum_{i\in \tilde{I}} \lambda_{i}\,,\qquad \partial_{\tilde{\Ga}}
= \sum_{i\in \tilde{I}} \partial_i\,,
\end{equation}
where $\partial_i$ denotes the action of $\partial$ on the $i$-th factor in the tensor product $V^{\otimes n}$. 
Then, for every $\alpha=1,\dots,s$,
\begin{equation}\label{eq:sesq1}
\frac{\partial}{\partial \lambda_i}
f^{\Gamma}_{\lambda_1,\dots,\lambda_n}(v_1\otimes\dots\otimes v_n)
\,\text{ is the same for all }\,
i\in I_\alpha
\,
\end{equation}
and 
\begin{equation}\label{eq:sesq2}
f^{\Gamma}_{\lambda_1,\dots,\lambda_n}
(\partial_{\Gamma_\alpha}(v_1\otimes\dots\otimes v_n))
=-\lambda_{\Gamma_\alpha}
f^{\Gamma}_{\lambda_1,\dots,\lambda_n}(v_1\otimes\dots\otimes v_n)
\,.
\end{equation}
Observe that the second sesquilinearity condition \eqref{eq:sesq2} implies
\begin{equation}\label{eq:sesq2.1}
f^{\Gamma}_{\lambda_1,\dots,\lambda_n}(\partial v) = 
-\sum_{i=1}^n \la_i \, f^{\Gamma}_{\lambda_1,\dots,\lambda_n}(v)
= \partial \bigl( f^{\Gamma}_{\lambda_1,\dots,\lambda_n}(v) \bigr), 
\qquad v\in V^{\otimes n}\,.
\end{equation}

\begin{remark}\label{rem:connected}
When the graph $\Gamma$ is connected, the first sesquilinearity condition \eqref{eq:sesq1} implies that
$f^{\Gamma}_{\lambda_1,\dots,\lambda_n}(v_1\otimes\dots\otimes v_n)$ is a polynomial of
$\lambda_1+\dots+\lambda_n$. Hence, it is an element of
$$
V[\lambda_1+\dots+\lambda_n]\big/\langle\partial+\lambda_1+\dots+\lambda_n\rangle
\simeq V \,.
$$
In this case, we will omit the subscripts of $f^{\Gamma}$.
\end{remark}

The classical operad $\mc P_{\cl}(V)$ is defined as the collection 
of the vector superspaces $\mc P_{\cl} (n)$, $n\ge 0$, endowed, for every $f\in \mc P_{\cl} (n)$ 
and $m_1,\dots, m_n \ge 0$, with the composition parity preserving linear map
\begin{equation*}
\begin{split}
\mc P_{\cl}(n) \otimes \mc P_{\cl}(m_1)\otimes\dots\otimes \mc P_{\cl}(m_n)\,\,&\to\,\,\mc P_{\cl}(M_n) \,, \\
f \otimes g_1 \otimes\dots\otimes g_n \,\, &\mapsto \,\, f(g_1,\dots,g_n) \,,
\end{split}
\end{equation*}
described as follows. Let $M_i$ be as in \eqref{eq:notation}, and 
\begin{equation} \label{eq:notation2}
\La_i = \sum_{j=M_{i-1}+1}^{M_i} \la_{j}\,, \qquad  i=1,\dots,n\,.
\end{equation}
If $\Gamma\in\mc G(M_n)$, then
\begin{equation*}
(f(g_1,\dots,g_n))^\Gamma:\, V^{\otimes M_n}
\to
V[\lambda_1,\dots,\lambda_{M_n}] \big/ \langle \partial+\lambda_1+\dots+\lambda_{M_n} \rangle
\end{equation*}
is defined by the formula:
\begin{equation}\label{eq:composition}
\begin{array}{l}
\displaystyle{
\vphantom{\Big(}
(f(g_1,\dots,g_n))^\Gamma_{\lambda_1,\dots,\lambda_{M_n}}
(v_1\otimes\dots\otimes v_{M_n})
} \\
\displaystyle{
\vphantom{\Big(}
=
f^{\Delta^{m_1\dots m_n}_0(\Gamma)}_{\Lambda_1,\dots,\Lambda_n}
\bigg(
\Big(
\Big(
\Big|_{x_1=\Lambda_1+\partial}
(g_1)^{\Delta^{m_1\dots m_n}_1(\Gamma)}_{\lambda_{1}+X(1),\dots,\lambda_{M_1}+X(M_1)}
\Big)
\otimes\cdots
} \\
\displaystyle{
\vphantom{\Big(}
\,\,\,\,\,\,\,\,\,
\dots\otimes
\Big(
\Big|_{x_n=\Lambda_n+\partial}
(g_n)^{\Delta^{m_1\dots m_n}_n(\Gamma)}_{
\lambda_{M_{n-1}+1}+X(M_{n-1}+1),\dots,\lambda_{M_n}+X(M_n)}
\Big)
\Big)
(v_1\otimes\dots\otimes v_{M_n})
\bigg)
}
\end{array}
\end{equation}
where $\Delta^{m_1,\dots,m_n}(\Ga)$ is the cocomposition of $\Ga$ described 
in Section \ref{ssec:ngraphs}, $X(1),\dots$ $\dots,X(M_n)$ are the variables as in \eqref{eq:X(k)}, 
and the notation is as follows. For given graphs $\Gamma_1\in\mc G(m_1),\,\dots,\Gamma_n\in\mc G(m_n)$, 
we have:
\begin{equation}\label{eq:tensorprodformula2}
\begin{split}
&\left( \left(g_1\right)^{\Gamma_1}_{\lambda_{1},\dots,\lambda_{M_1}}
 \otimes\dots\otimes 
\left(g_n\right)^{\Gamma_n}_{\lambda_{M_{n-1}+1},\dots,\lambda_{M_n}} \right) 
(v_1\otimes\dots\otimes v_{M_n}) \\
&:= (-1)^{\sum_{i<j} p(g_j)(p(v_{M_{i-1}+1})+ \dots + p(v_{M_i}))} \,
(g_1)^{\Ga_1}_{\lambda_{1},\dots,\lambda_{M_1}}\!\!(v_1\otimes\dots\otimes v_{M_1})
\otimes\cdots \\
&\qquad \qquad\qquad\qquad\qquad\qquad \dots \otimes (g_n)^{\Gamma_n}_{\lambda_{M_{\!n\!-\!1}\!+1},
\dots,\lambda_{M_n}}\!\!(v_{M_{\!n\!-\!1}\!+1}\otimes\dots\otimes v_{M_n}),
\end{split}	
\end{equation}
and for polynomials $P(\lambda)=\sum_mp_m\lambda^m$ and 
$Q(\mu)=\sum_nq_n\mu^n$ with coefficients in $V$,
we write
\begin{equation}\label{eq:polynomials}
\big(\big|_{x=\partial}P(\lambda+y)\big)
\otimes
\big(\big|_{y=\partial}Q(\mu+x)\big)
=
\sum_{m,n}((\mu+\partial)^np_m)\otimes((\lambda+\partial)^mq_n)\,.
\end{equation}

For each $n\ge 1$, $\mc P_{\cl}(n)$ has a natural right action of the symmetric group $S_n$,
which is given by ($f\in \mc P_{\cl}(n)$, $\Gamma\in\mc G(n)$, $v_1,\dots,v_n\in V$):
\begin{equation}\label{eq:actclassicoperad}
(f^\sigma)^\Gamma_{\lambda_1,\dots,\lambda_n}(v_1\otimes\cdots\otimes v_n)
=
f^{\sigma(\Gamma)}_{\sigma(\lambda_1,\dots,\lambda_n)}
(\sigma(v_1\otimes\cdots\otimes v_n))
\,,  
\end{equation}
where 
$\sigma(\lambda_1,\dots,\lambda_n)$ is defined by \eqref{eq:acttuple},
$\sigma(v_1\otimes\dots\otimes v_n)$ is defined by \eqref{eq:acttensorprod},
and $\sigma(\Gamma)$ is defined in Section \ref{ssec:ngraphs}. 

On the space $\mc P_{\cl} (n)$ we can also define a grading:
\begin{equation} \label{eq:clgrading}
	\mc P_{\cl} (n) = \bigoplus _{r\geq 0} \gr^r \mc P_{\cl} (n)\,,
\end{equation}
where $\gr^r \mc P_{\cl} (n)$ is the subspace of all maps in $\mc P_{\cl} (n)$ vanishing on graphs with a number 
of edges not equal to $r$. Then $\mc P_{\cl}$ is a graded operad, i.e., the compositions and the actions of the 
symmetric groups are compatible with the grading.

\subsection{PVA cohomology \cite{BDSHK18}} \label{ssec:PVAcohom}

Given a vector superspace $V$ with parity $p$, 
and an even endomorphism $\partial\in \End(V)$, 
let $\Pi V$ be the same vector space with reversed parity $\bar{p}=1-p$, 
and consider the corresponding classical operad $\mc P_{\cl} (\Pi V)$ from Section \ref{ssec:cloperad}. 
The associated $\mb Z$-graded Lie superalgebra is $W_\cl (\Pi V):= W(\mc P_{\cl} (\Pi V))$, 
with Lie bracket defined by \eqref{eq:liebracket}.
\begin{theorem}[{\cite[Theorem 10.7]{BDSHK18}}]\label{thm:PVAstructure}
We have a bijective correspondence between 
the odd elements $X\in W^1_{\cl}(\Pi V)$, such that $X\Box X=0$,
and the Poisson vertex superalgebra structures on $V$,
defined as follows.
The  commutative associative product and the $\lambda$-bracket 
of the Poisson vertex superalgebra $V$ corresponding to $X$
are given by
\begin{equation}\label{eq:PVAstructure}
ab=(-1)^{p(a)}
X^{\bullet\!-\!\!\!\!\to\!\bullet}(a\otimes b)
\,\,,\,\,\,\,
[a_\lambda b]=(-1)^{p(a)}X^{\bullet\,\,\bullet}_{\lambda,-\lambda-\partial}(a\otimes b)
\,.
\end{equation}
\end{theorem}
Thanks to the Jacobi identity for the Lie superalgebra $W_\cl(\Pi V)$, if $X\in W^1_\cl(\Pi V)_{\bar{1}}$ 
satisfies $X\Box X=0$, then $(\ad X)^2=0$. In view of Theorem \ref{thm:PVAstructure}, this means that 
we have a cohomology complex
\begin{equation*}
(W_{\cl}(\Pi V),\ad X)\,,
\end{equation*}
called the
\emph{PVA cohomology complex}, where $X\in W^1(\Pi V)_{\bar 1}$ is given by \eqref{eq:PVAstructure}.
 
\section{Relation between PVA and differential Harrison cohomology complexes} \label{sec:main}

\subsection{Main theorem}

Let $V$ be a Poisson vertex algebra. By Theorem \ref{thm:PVAstructure}, 
we have an odd element $X\in W_\cl (\Pi V)$ such that $[X,X]=0$, 
which is associated to the PVA structure of $V$ by \eqref{eq:PVAstructure}. 
Thus, there is the PVA cohomology complex 
\begin{equation} \label{eq:complex1}
\left(W_\cl(\Pi V),\ad X\right)\,.
\end{equation}
A classical $n$-cochain is an element $Y\in W^{n-1}_{\cl}(\Pi V)$, namely a map
\begin{equation}\label{eq:clcochain}
Y\colon \mc G(n)\times (\Pi V)^{\otimes n} \longrightarrow 
(\Pi V)[\lambda_1,\dots,\lambda_n] \big/ \langle \partial+\lambda_1+\dots+\lambda_n \rangle
\,,
\end{equation}
satisfying relations \eqref{eq:cycle1},  \eqref{eq:cycle2},  \eqref{eq:sesq1}, \eqref{eq:sesq2}, 
and the following symmetry property (by definition \eqref{eq:liesuperalg}):
\begin{equation} \label{eq:symmetry}
Y^\sigma=Y\,,\quad \forall\,\sigma\in S_{n}\;.
\end{equation}

Recall the grading of the superoperad $\mc P_{\cl} (\Pi V)$ from \eqref{eq:clgrading}: $\gr^r\,W^{n-1}_\cl(\Pi V)$ 
is the set of maps $Y$ as in \eqref{eq:clcochain} such that 
$$
Y^{\Ga}=0 \;\;\text{ unless }\;\; |E(\Ga)| =r\,.
$$
Note that if $\Ga\in \mc G(n)$ has $| E(\Ga)| \geq n$, then necessarily $\Ga$ contains a cycle. 
Hence, by the cycle relation \eqref{eq:cycle1}, $Y^{\Ga}=0$. Therefore the top degree 
in $\gr W^{n-1}_{\cl}(\Pi V)$ is $n-1$, i.e.
\begin{equation}\label{eq:maxgr}
\gr^{r}\,W^{n-1}_\cl(\Pi V)=0 \;\;\text{ if }\;\; r\geq n\,.
\end{equation} 
Note that, if $\Ga\in \mc G_{0}(n)$, then $| E(\Ga)| =n-1$ if and only if $\Ga$ is connected.
By Remark \ref{rem:connected},
the top degree subspace $\gr^{n-1}W^{n-1}_{\cl}(\Pi V)$ consists of all collections of maps
\begin{equation}\label{eq:number}
Y^{\Ga}\colon (\Pi V)^{\otimes n} \longrightarrow (\Pi V)\,, \quad\text{ for }\; \Ga\in \mc G_{0}(n),\;\; |E(\Ga)|=n-1 \,,
\end{equation}
satisfying \eqref{eq:cycle1},  \eqref{eq:cycle2}, \eqref{eq:symmetry}, 
and $Y^{\Ga}(\partial(v_1\otimes\dots\otimes v_n))=\partial Y^{\Ga}(v_1\otimes\dots\otimes v_n)$. 
If $\Ga$ is not connected, then $Y^{\Ga}=0$.

In addition, as explained in Section \ref{ssec:Harrison}, there is another cohomology complex associated 
to $V$, viewed as a differential algebra, namely the differential Harrison complex
\begin{equation} \label{eq:complex2}
(C_{\partial,\Har}(V),d)\,,
\end{equation}
where $C^n_{\partial,\Har}(V)\subset \Hom_{\mb F[\partial]}(V^{\otimes n},V)$ is defined by 
Harrison's conditions \eqref{eq:harrisoncond} and $d$ is 
the Hochschild differential \eqref{eq:hochdifferential}.

The main result of this paper is the following:
\begin{theorem} \label{thm:main}
Let\/ $V$ be a Poisson vertex algebra. 
One has a surjective morphism of cochain complexes 
\begin{equation} \label{eq:morphism}
(W_{\cl}(\Pi V), \ad X) \rightarrow (C_{\partial,\Har}(V),d)\,,
\end{equation}
mapping\/ $Y\in W^{n-1}_{\cl}(\Pi V)$ to\/ $Y^{\La_n}$, where\/ $\La_n$ is the standard $n$-line
\begin{equation} \label{eq:Ln}
\begin{tikzpicture}[scale=1.5]
\node at (-0.5,0){$\La_n=$};
\draw[fill] (0,0) circle [radius=0.02];
\node at (0,-0.2){\tiny{$1$}};
\draw[->] (0,0) to (0.3,0);
\draw[fill] (0.3,0) circle [radius=0.02];
\node at (0.3,-0.2){\tiny{$2$}};
\draw[->] (0.3,0) to (0.6,0);
\node at (0.7,0) {.};
\node at (0.75,0) {.};
\node at (0.8,0) {.};
\draw[->] (0.9,0) to (1.2,0);
\draw[fill] (1.2,0) circle [radius=0.02];
\node at (1.2,-0.2){\tiny{$n$}};
\end{tikzpicture}
\end{equation}
Moreover, the morphism \eqref{eq:morphism} restricts to a bijective linear map on the top degree:
\begin{equation} \label{eq:topmorphism}
\gr^{n-1}W^{n-1}_{\cl}(\Pi V) \xrightarrow[]{\sim} C^n_{\partial,\Har}(V)\,.
\end{equation}
\end{theorem}
We will prove Theorem \ref{thm:main} in Section \ref{ssec:proof}. 
For that, we will need some preliminary results. 

\subsection{Lines} \label{ssec:lines}

We say that a graph $\Ga\in \mc G(n)$ is a non-connected line if it has the following form:
\begin{equation}\label{eq:unionlines}
\begin{tikzpicture}
\node at (-0.2,1) {$\Gamma=$};
\draw[fill] (0.5,1) circle [radius=0.07];
\node at (0.5,0.6) {$i^1_1$};
\draw[->] (0.6,1) -- (0.9,1);
\draw[fill] (1,1) circle [radius=0.07];
\node at (1,0.6) {$i^1_2$};
\draw[->] (1.1,1) -- (1.4,1);
\node at (1.7,1) {$\cdots$};
\draw[->] (1.9,1) -- (2.2,1);
\draw[fill] (2.3,1) circle [radius=0.07];
\node at (2.3,0.6) {$i^1_{k_1}$};
\draw[fill] (3,1) circle [radius=0.07];
\node at (3,0.6) {$i^2_1$};
\draw[->] (3.1,1) -- (3.4,1);
\draw[fill] (3.5,1) circle [radius=0.07];
\node at (3.5,0.6) {$i^2_2$};
\draw[->] (3.6,1) -- (3.9,1);
\node at (4.2,1) {$\cdots$};
\draw[->] (4.4,1) -- (4.7,1);
\draw[fill] (4.8,1) circle [radius=0.07];
\node at (4.8,0.6) {$i^2_{k_2}$};
\node at (5.45,1) {$\cdots$};
\draw[fill] (6,1) circle [radius=0.07];
\node at (6,0.6) {$i^s_1$};
\draw[->] (6.1,1) -- (6.4,1);
\draw[fill] (6.5,1) circle [radius=0.07];
\node at (6.5,0.6) {$i^s_2$};
\draw[->] (6.6,1) -- (6.9,1);
\node at (7.2,1) {$\cdots$};
\draw[->] (7.4,1) -- (7.7,1);
\draw[fill] (7.8,1) circle [radius=0.07];
\node at (7.8,0.6) {$i^s_{k_s}$};
\node at (9.7,1) {$=L_1\sqcup L_2\sqcup\dots\sqcup L_s\,,$};
\end{tikzpicture}
\end{equation}
where $1\leq k_1 \leq \dots \leq k_s$ are such that $k_1+\dots+k_s=n$,
and the set of indices $\{i^a_b\}$ is a permutation of $\{1,\dots,n\}$ such that
\begin{equation}
i^l_1=\min\{i_1^l,\dots,i_{k_l}^l\} \;\;\;\forall\;l=1,\dots,s\,.
\end{equation}
If $k_l=k_{l+1}$, we also assume that $i_1^l<i_1^{l+1}$. In particular, the connected lines are all of the form
\begin{equation}\label{eq:lines}
\sigma (\La_n)\,, \;\;\; \sigma\in S_n\,\text{ such that } \sigma(1)=1\,,
\end{equation}
where $\La_n$ is the $n$-line \eqref{eq:Ln}. Let $\mathcal{L}(n)$ be the set of $n$-graphs 
that are non-connected lines. Let also $\mb F\mc G (n)$ be the vector space with basis 
the set of graphs $\mc{G}(n)$. 
\begin{definition}
The \emph{cycle relations} in $\mb{F}\mc G (n)$ are the following elements:
\begin{enumerate}[(i)]
\item
all $\Ga\in \mc G(n)\setminus \mc{G}_{0}(n)$ (i.e., graphs containing a cycle); \label{eq:cyclerel1}
\item
all linear combinations $\sum_{e\in C} \Ga\setminus e$, where $\Ga\in \mc{G}(n)$ 
and $C\subset E(\Ga)$ is an oriented cycle. \label{eq:cyclerel2}
\end{enumerate}
\end{definition}
Denote by $R(n)\subset \mb{F}\mc{G}(n)$ the subspace spanned by the cycle relations \eqref{eq:cyclerel1} 
and \eqref{eq:cyclerel2}. 

Note that reversing an arrow in a graph $\Ga\in \mc{G}(n)$ gives us, modulo cycle relations, 
the element $-\Ga\in \mb{F}\mc{G}(n)$.
\begin{example}\label{ex:triangle}
For $n=3$, a cycle relation of type \eqref{eq:cyclerel2} is:
\begin{equation}\label{eq:triangle}
\begin{tikzpicture}
\node at (0,0) {$2$};
\node at (1.4,0) {$3$};
\node at (0.7,1) {$1$};
\draw[->] (0.6,0.9) to (0.1,0.1);
\draw[->] (0.1,0) to (1.3,0);
\end{tikzpicture}
\genfrac{}{}{0pt}{}{\;\;+\;\;}{}
\begin{tikzpicture}
\node at (0,0) {$2$};
\node at (1.4,0) {$3$};
\node at (0.7,1) {$1$};
\draw[->] (0.1,0) to (1.3,0);
\draw[<-] (0.8,0.9) to (1.3,0.1);
\end{tikzpicture}
\genfrac{}{}{0pt}{}{\;\;+\;\;}{}
\begin{tikzpicture}
\node at (0,0) {$2$};
\node at (1.4,0) {$3$};
\node at (0.7,1) {$1$};
\draw[<-] (0.8,0.9) to (1.3,0.1);
\draw[->] (0.6,0.9) to (0.1,0.1);
\end{tikzpicture}
\end{equation}
\end{example}
\begin{remark}\label{rem:cycle}
The cycle relations \eqref{eq:cycle1} and \eqref{eq:cycle2} on $Y\in \mc{P}_{\cl}$
can be restated by saying that $Y^{\Ga}=0$ for all $\Ga\in R(n)$.
\end{remark}
\begin{theorem}[{\cite[Theorem\ 4.7]{BDSHK19}}] \label{thm:lines}
The set\/ $\mathcal{L}(n)$ is a basis for the quotient space\/ $\mb{F}\mc{G}(n)/R(n)$.
\end{theorem}

By Theorem \ref{thm:lines} and \eqref{eq:lines}, we can write every connected graph $\Ga\in \mc G(n)$, 
uniquely, modulo cycle relation, as follows:
\begin{equation} \label{eq:lincomb}
\Ga \equiv \sum_{\substack{\sigma\in S_n\\ \sigma(1)=1}} c_{\sigma}^{\Ga} \,\sigma \La_n\,,
\end{equation}
where $c_{\sigma}^{\Ga}\in \mb{F}$ and the action of the symmetric group on graphs 
is defined in Section \ref{ssec:ngraphs}.
Here and further, by $\equiv$ we mean equivalence modulo cycle relations, 
i.e., equality in the quotient space $\mb{F}\mc{G}(n)/R(n)$.

\subsection{Connected lines} \label{ssec:connlines}

\begin{lemma} \label{lem:connline}
For every $n$, the following identity on $n$-lines holds:
\begin{equation}\label{eq:connline}
\begin{tikzpicture}[scale=0.8]
\draw[fill] (0,0) circle [radius=0.04];
\node at (0,-0.2) {\tiny{$1$}};
\draw[fill] (0.5,0) circle [radius=0.04];
\node at (0.5,-0.2) {\tiny{$2$}};
\node at (1,0) {.};
\node at (1.1,0) {.};
\node at (1.2,0) {.};
\draw[fill] (1.7,0) circle [radius=0.04];
\node at (1.7,-0.2) {\tiny{$n$}};
\draw[->] (0.1,0) to (0.4,0);
\draw[->] (0.6,0) to (0.9,0);	
\draw[->] (1.3,0) to (1.6,0);
\end{tikzpicture}
\genfrac{}{}{0pt}{}{+}{}
\begin{tikzpicture}
\draw[fill] (0,0) circle [radius=0.04];
\node at (0,-0.2) {\tiny{$2$}};
\draw[fill] (0.5,0) circle [radius=0.04];
\node at (0.5,-0.2) {\tiny{$1$}};
\draw[fill] (1,0) circle [radius=0.04];
\node at (1,-0.2) {\tiny{$3$}};
\node at (1.5,0) {.};
\node at (1.6,0) {.};
\node at (1.7,0) {.};
\draw[fill] (2.2,0) circle [radius=0.04];
\node at (2.2,-0.2) {\tiny{$n$}};
\draw[->] (0.1,0) to (0.4,0);
\draw[->] (0.6,0) to (0.9,0);	
\draw[->] (1.1,0) to (1.4,0);
\draw[->] (1.8,0) to (2.1,0);
\end{tikzpicture}
\genfrac{}{}{0pt}{}{+}{}
\genfrac{}{}{0pt}{}{\cdots}{}
\genfrac{}{}{0pt}{}{+}{}
\begin{tikzpicture}
\draw[fill] (0,0) circle [radius=0.04];
\node at (0,-0.2) {\tiny{$2$}};
\draw[fill] (0.5,0) circle [radius=0.04];
\node at (0.5,-0.2) {\tiny{$3$}};
\node at (1,0) {.};
\node at (1.1,0) {.};
\node at (1.2,0) {.};
\draw[fill] (1.7,0) circle [radius=0.04];
\node at (1.7,-0.2) {\tiny{$n\!-\!1$}};
\draw[fill] (2.2,0) circle [radius=0.04];
\node at (2.2,-0.2) {\tiny{$1$}};
\draw[fill] (2.7,0) circle [radius=0.04];
\node at (2.7,-0.2) {\tiny{$n$}};
\draw[->] (0.1,0) to (0.4,0);
\draw[->] (0.6,0) to (0.9,0);	
\draw[->] (1.3,0) to (1.6,0);
\draw[->] (1.8,0) to (2.1,0);
\draw[->] (2.3,0) to (2.6,0);
\end{tikzpicture}
\genfrac{}{}{0pt}{}{+}{}
\begin{tikzpicture}
\draw[fill] (0,0) circle [radius=0.04];
\node at (0,-0.2) {\tiny{$2$}};
\draw[fill] (0.5,0) circle [radius=0.04];
\node at (0.5,-0.2) {\tiny{$3$}};
\node at (1,0) {.};
\node at (1.1,0) {.};
\node at (1.2,0) {.};
\draw[fill] (1.7,0) circle [radius=0.04];
\node at (1.7,-0.2) {\tiny{$n$}};
\draw[fill] (2.2,0) circle [radius=0.04];
\node at (2.2,-0.2) {\tiny{$1$}};
\draw[->] (0.1,0) to (0.4,0);
\draw[->] (0.6,0) to (0.9,0);	
\draw[->] (1.3,0) to (1.6,0);
\draw[->] (1.8,0) to (2.1,0);
\end{tikzpicture}
\genfrac{}{}{0pt}{}{\equiv}{}
\genfrac{}{}{0pt}{}{0\,.}{}
\end{equation}
\end{lemma}
\begin{proof}
Let us consider the first two terms in the left-hand side of \eqref{eq:connline}. 
Reversing the edges $2\rightarrow 1$ and $1\rightarrow 3$ in the second graph, 
and using \eqref{eq:triangle}, we have:
\begin{equation} \label{eq:p1}
\begin{tikzpicture}
\node at (0,0) {\tiny{$2$}};
\node at (0.25,0.5) {\tiny{$1$}};
\node at (0.5,0) {\tiny{$3$}};
\node at (1,0) {.};
\node at (1.1,0) {.};
\node at (1.2,0) {.};
\node at (1.7,0) {\tiny{$n$}};
\draw[->] (0.2,0.45) to (0,0.1);
\draw[->] (0.1,0) to (0.4,0);
\draw[->] (0.6,0) to (0.9,0);	
\draw[->] (1.3,0) to (1.6,0);
\end{tikzpicture}
\genfrac{}{}{0pt}{}{+}{}
\begin{tikzpicture}
\node at (0,0) {\tiny{$2$}};
\node at (0.25,0.5) {\tiny{$1$}};
\node at (0.5,0) {\tiny{$3$}};
\node at (1,0) {\tiny{$4$}};
\node at (1.5,0) {.};
\node at (1.6,0) {.};
\node at (1.7,0) {.};
\node at (2.2,0) {\tiny{$n$}};
\draw[<-] (0.3,0.45) to (0.5,0.1);
\draw[->] (0.2,0.45) to (0,0.1);
\draw[->] (0.6,0) to (0.9,0);	
\draw[->] (1.1,0) to (1.4,0);
\draw[->] (1.8,0) to (2.1,0);
\end{tikzpicture}
\genfrac{}{}{0pt}{}{\equiv}{}
\genfrac{}{}{0pt}{}{-}{}
\begin{tikzpicture}
\node at (0,0) {\tiny{$2$}};
\node at (0.25,0.5) {\tiny{$1$}};
\node at (0.5,0) {\tiny{$3$}};
\node at (1,0) {.};
\node at (1.1,0) {.};
\node at (1.2,0) {.};
\node at (1.7,0) {\tiny{$n$}};
\draw[<-] (0.3,0.45) to (0.5,0.1);
\draw[->] (0.1,0) to (0.4,0);
\draw[->] (0.6,0) to (0.9,0);	
\draw[->] (1.3,0) to (1.6,0);
\end{tikzpicture}
\genfrac{}{}{0pt}{}{\equiv}{}
\begin{tikzpicture}
\node at (0,0) {\tiny{$2$}};
\node at (0.75,0.5) {\tiny{$1$}};
\node at (0.5,0) {\tiny{$3$}};
\node at (1,0) {\tiny{$4$}};
\node at (1.5,0) {.};
\node at (1.6,0) {.};
\node at (1.7,0) {.};
\node at (2.2,0) {\tiny{$n$}};
\draw[->] (0.7,0.45) to (0.5,0.1);
\draw[->] (0.1,0) to (0.4,0);
\draw[->] (0.6,0) to (0.9,0);
\draw[->] (1.1,0) to (1.4,0);
\draw[->] (1.8,0) to (2.1,0);
\end{tikzpicture}
\end{equation}
Adding \eqref{eq:p1} to the third term appearing in the left-hand side of \eqref{eq:connline}, and applying again \eqref{eq:triangle}, we obtain:
\begin{equation} \label{eq:p2}
\begin{tikzpicture}
\node at (-0.5,0) {\tiny{$2$}};
\node at (0,0) {\tiny{$3$}};
\node at (0.25,0.5) {\tiny{$1$}};
\node at (0.5,0) {\tiny{$4$}};
\node at (1,0) {.};
\node at (1.1,0) {.};
\node at (1.2,0) {.};
\node at (1.7,0) {\tiny{$n$}};
\draw[->] (-0.4,0) to (-0.1,0);
\draw[->] (0.2,0.45) to (0,0.1);
\draw[->] (0.1,0) to (0.4,0);
\draw[->] (0.6,0) to (0.9,0);	
\draw[->] (1.3,0) to (1.6,0);
\end{tikzpicture}
\genfrac{}{}{0pt}{}{+}{}
\begin{tikzpicture}
\node at (-0.5,0) {\tiny{$2$}};
\node at (0,0) {\tiny{$3$}};
\node at (0.25,0.5) {\tiny{$1$}};
\node at (0.5,0) {\tiny{$4$}};
\node at (1,0) {.};
\node at (1.1,0) {.};
\node at (1.2,0) {.};
\node at (1.7,0) {\tiny{$n$}};
\draw[->] (-0.4,0) to (-0.1,0);
\draw[->] (0.2,0.45) to (0,0.1);
\draw[<-] (0.3,0.45) to (0.5,0.1);
\draw[->] (0.6,0) to (0.9,0);	
\draw[->] (1.3,0) to (1.6,0);
\end{tikzpicture}
\genfrac{}{}{0pt}{}{\equiv}{}
\genfrac{}{}{0pt}{}{-}{}
\begin{tikzpicture}
\node at (-0.5,0) {\tiny{$2$}};
\node at (0,0) {\tiny{$3$}};
\node at (0.25,0.5) {\tiny{$1$}};
\node at (0.5,0) {\tiny{$4$}};
\node at (1,0) {.};
\node at (1.1,0) {.};
\node at (1.2,0) {.};
\node at (1.7,0) {\tiny{$n$}};
\draw[->] (-0.4,0) to (-0.1,0);
\draw[<-] (0.3,0.45) to (0.5,0.1);
\draw[->] (0.1,0) to (0.4,0);
\draw[->] (0.6,0) to (0.9,0);	
\draw[->] (1.3,0) to (1.6,0);
\end{tikzpicture}
\genfrac{}{}{0pt}{}{\equiv}{}
\begin{tikzpicture}
\node at (-0.5,0) {\tiny{$2$}};
\node at (0,0) {\tiny{$3$}};
\node at (0.75,0.5) {\tiny{$1$}};
\node at (0.5,0) {\tiny{$4$}};
\node at (1,0) {.};
\node at (1.1,0) {.};
\node at (1.2,0) {.};
\node at (1.7,0) {\tiny{$n$}};
\draw[->] (-0.4,0) to (-0.1,0);
\draw[->] (0.7,0.45) to (0.5,0.1);
\draw[->] (0.1,0) to (0.4,0);
\draw[->] (0.6,0) to (0.9,0);	
\draw[->] (1.3,0) to (1.6,0);
\end{tikzpicture}
\end{equation}
We proceed in the same way, up to 
\begin{equation}
\begin{tikzpicture}
\node at (0,0) {\tiny{$2$}};
\node at (0.5,0) {\tiny{$3$}};
\node at (1,0) {.};
\node at (1.1,0) {.};
\node at (1.2,0) {.};
\node at (1.9,0) {\tiny{$(\!n\!-\!1\!)$}};
\node at (2.6,0) {\tiny{$n$}};
\node at (2.25,0.5) {\tiny{$1$}};
\draw[->] (0.1,0) to (0.4,0);
\draw[->] (0.6,0) to (0.9,0);	
\draw[->] (1.3,0) to (1.6,0);
\draw[->] (2.2,0) to (2.5,0);
\draw[->] (2.2,0.4) to (1.9,0.1);
\end{tikzpicture}
\genfrac{}{}{0pt}{}{+}{}
\begin{tikzpicture}
\node at (0,0) {\tiny{$2$}};
\node at (0.5,0) {\tiny{$3$}};
\node at (1,0) {.};
\node at (1.1,0) {.};
\node at (1.2,0) {.};
\node at (1.9,0) {\tiny{$(\!n\!-\!1\!)$}};
\node at (2.6,0) {\tiny{$n$}};
\node at (2.25,0.5) {\tiny{$1$}};
\draw[->] (0.1,0) to (0.4,0);
\draw[->] (0.6,0) to (0.9,0);	
\draw[->] (1.3,0) to (1.6,0);
\draw[->] (2.2,0.4) to (1.9,0.1);
\draw[<-] (2.3,0.4) to (2.6,0.1);
\end{tikzpicture}
\genfrac{}{}{0pt}{}{+}{}
\begin{tikzpicture}
\node at (0,0) {\tiny{$2$}};
\node at (0.5,0) {\tiny{$3$}};
\node at (1,0) {.};
\node at (1.1,0) {.};
\node at (1.2,0) {.};
\node at (1.9,0) {\tiny{$(\!n\!-\!1\!)$}};
\node at (2.6,0) {\tiny{$n$}};
\node at (2.25,0.5) {\tiny{$1$}};
\draw[->] (0.1,0) to (0.4,0);
\draw[->] (0.6,0) to (0.9,0);	
\draw[->] (1.3,0) to (1.6,0);
\draw[->] (2.2,0) to (2.5,0);
\draw[<-] (2.3,0.4) to (2.6,0.1);
\end{tikzpicture}
\genfrac{}{}{0pt}{}{\equiv}{}
\genfrac{}{}{0pt}{}{0 \,.}{}
\end{equation}
This completes the proof of the lemma.
\end{proof}
\begin{remark}\label{rem:local}
Observe that equation \eqref{eq:connline} can be viewed as a ``local'' identity: Lemma \ref{lem:connline} 
holds even if we attach the same graph $\Ga$ at any vertex of every graph appearing in the identity. 
\end{remark}
\begin{lemma} \label{lem:identity}
Let\/ $\La_n$ be as in \eqref{eq:Ln}. For every $k\in \{2,\dots,n\}$, the following identity holds:
\begin{equation} \label{eq:identity}
\La_n + (-1)^{k} \sum_{\pi\in \mc M_n^k} \pi \La_n \equiv 0\,,
\end{equation} 
where the sum is over all monotone permutations $\pi$ starting at $k$ and the action of $S_n$
on graphs is described in Section \ref{ssec:ngraphs}.
\end{lemma}
\begin{proof}
If $k=2$, then \eqref{eq:identity} is equivalent to \eqref{eq:connline}.
Suppose that $k>2$. Recall the description \eqref{eq:correspondence} of the monotone permutations 
$\pi\in \mc M_n^k$ in terms of the set $D(\pi)$ of drops. Given the set of drops 
$D=\{d_1,\dots,d_{k-1}\}$ such that
$2\leq d_{k-1}<\dots <d_1\leq n$, the corresponding monotone permutation $\pi^D\in \mc M_n^k$ 
is uniquely determined by $\pi^{-1}(i)=d_i$, $\forall i=1\dots, k-1$. Hence:
\begin{equation*}
\begin{tikzpicture}
\node at (-1.2,0) {$\pi^D (\La_n) =$};
\node at (0,0) {\tiny{$k$}};
\node at (0.7,0) {\tiny{$(\!k\!+\!1\!)$}};
\node at (1.4,0) {.};
\node at (1.5,0) {.};
\node at (1.6,0) {.};
\node at (2.3,0) {\tiny{$\underline{(\!k\!-\!1\!)}$}};
\node at (2.3,-0.5) {\tiny{$[d_{k-1}]$}};
\node at (3,0) {.};
\node at (3.1,0) {.};
\node at (3.2,0) {.};
\node at (3.7,0) {\tiny{$\underline{2}$}};
\node at (3.7,-0.5) {\tiny{$[d_{2}]$}};
\node at (4.2,0) {.};
\node at (4.3,0) {.};
\node at (4.4,0) {.};
\node at (4.9,0) {\tiny{$\underline{1}$}};
\node at (4.9,-0.5) {\tiny{$[d_{1}]$}};
\node at (5.4,0) {.};
\node at (5.5,0) {.};
\node at (5.6,0) {.};
\node at (6.1,0) {\tiny{$n$}};
\draw[->] (0.1,0) to (0.4,0);
\draw[->] (1,0) to (1.3,0);
\draw[->] (1.7,0) to (2,0);
\draw[->] (2.6,0) to (2.9,0);
\draw[->] (3.3,0) to (3.6,0);
\draw[->] (3.8,0) to (4.1,0);
\draw[->] (4.5,0) to (4.8,0);
\draw[->] (5,0) to (5.3,0);
\draw[->] (5.7,0) to (6,0);
\end{tikzpicture}
\end{equation*}
where the underlined positions correspond to drops, while all other positions have vertices in increasing order from $k+1$ to $n$. 

We then have:
\begin{equation*} 
\begin{tikzpicture}
\node at (-2.5,-1.2) {$\displaystyle\sum\limits_{\pi\in \mc M_n^k} \pi \La_n = \displaystyle\sum\limits_{2\leq d_{k-1}<\dots <d_1\leq n}$};
\node at (0,-1) {\tiny{$k$}};
\node at (0.7,-1) {\tiny{$(\!k\!+\!1\!)$}};
\node at (1.4,-1) {.};
\node at (1.5,-1) {.};
\node at (1.6,-1) {.};
\node at (2.3,-1) {\tiny{$\underline{(\!k\!-\!1\!)}$}};
\node at (2.3,-1.5) {\tiny{$[d_{k-1}]$}};
\node at (3,-1) {.};
\node at (3.1,-1) {.};
\node at (3.2,-1) {.};
\node at (3.7,-1) {\tiny{$\underline{2}$}};
\node at (3.7,-1.5) {\tiny{$[d_{2}]$}};
\node at (4.2,-1) {.};
\node at (4.3,-1) {.};
\node at (4.4,-1) {.};
\node at (4.9,-1) {\tiny{$\underline{1}$}};
\node at (4.9,-1.5) {\tiny{$[d_{1}]$}};
\node at (5.4,-1) {.};
\node at (5.5,-1) {.};
\node at (5.6,-1) {.};
\node at (6.1,-1) {\tiny{$n$}};
\draw[->] (0.1,-1) to (0.4,-1);
\draw[->] (1,-1) to (1.3,-1);
\draw[->] (1.7,-1) to (2,-1);
\draw[->] (2.6,-1) to (2.9,-1);
\draw[->] (3.3,-1) to (3.6,-1);
\draw[->] (3.8,-1) to (4.1,-1);
\draw[->] (4.5,-1) to (4.8,-1);
\draw[->] (5,-1) to (5.3,-1);
\draw[->] (5.7,-1) to (6,-1);
\node at (6.5,-1.2) {.};
\end{tikzpicture}
\end{equation*}
By Lemma \ref{lem:connline}, the sum over $d_1\in \{d_2+1,\dots,n\}$ gives
\begin{align*}
&\begin{tikzpicture}
\node at (-0.9,-1) {$\displaystyle\sum\limits_{d_1=d_2+1}^{n}$};
\node at (0,-1) {\tiny{$k$}};
\node at (0.7,-1) {\tiny{$(\!k\!+\!1\!)$}};
\node at (1.4,-1) {.};
\node at (1.5,-1) {.};
\node at (1.6,-1) {.};
\node at (2.3,-1) {\tiny{$\underline{(\!k\!-\!1\!)}$}};
\node at (2.3,-1.5) {\tiny{$[d_{k-1}]$}};
\node at (3,-1) {.};
\node at (3.1,-1) {.};
\node at (3.2,-1) {.};
\node at (3.7,-1) {\tiny{$\underline{2}$}};
\node at (3.7,-1.5) {\tiny{$[d_{2}]$}};
\node at (4.2,-1) {.};
\node at (4.3,-1) {.};
\node at (4.4,-1) {.};
\node at (4.9,-1) {\tiny{$\underline{1}$}};
\node at (4.9,-1.5) {\tiny{$[d_{1}]$}};
\node at (5.4,-1) {.};
\node at (5.5,-1) {.};
\node at (5.6,-1) {.};
\node at (6.1,-1) {\tiny{$n$}};
\draw[->] (0.1,-1) to (0.4,-1);
\draw[->] (1,-1) to (1.3,-1);
\draw[->] (1.7,-1) to (2,-1);
\draw[->] (2.6,-1) to (2.9,-1);
\draw[->] (3.3,-1) to (3.6,-1);
\draw[->] (3.8,-1) to (4.1,-1);
\draw[->] (4.5,-1) to (4.8,-1);
\draw[->] (5,-1) to (5.3,-1);
\draw[->] (5.7,-1) to (6,-1);
%
\end{tikzpicture} \notag \\
&\begin{tikzpicture}
\node at (-0.5,0) {$=-$};
\node at (0,0) {\tiny{$k$}};
\node at (0.7,0) {\tiny{$(\!k\!+\!1\!)$}};
\node at (1.4,0) {.};
\node at (1.5,0) {.};
\node at (1.6,0) {.};
\node at (2.3,0) {\tiny{$\underline{(\!k\!-\!1\!)}$}};
\node at (2.3,-0.5) {\tiny{$[d_{k-1}]$}};
\node at (3,0) {.};
\node at (3.1,0) {.};
\node at (3.2,0) {.};
\node at (3.7,0) {\tiny{$\underline{2}$}};
\node at (3.7,-0.5) {\tiny{$[d_{2}]$}};
\node at (3.7,0.5) {\tiny{$1$}};
\node at (4.2,0) {.};
\node at (4.3,0) {.};
\node at (4.4,0) {.};
\node at (4.9,0) {\tiny{$n$}};
\draw[->] (0.1,0) to (0.4,0);
\draw[->] (1,0) to (1.3,0);
\draw[->] (1.7,0) to (2,0);
\draw[->] (2.6,0) to (2.9,0);
\draw[->] (3.3,0) to (3.6,0);
\draw[->] (3.8,0) to (4.1,0);
\draw[->] (4.5,0) to (4.8,0);
\draw[->] (3.7,0.4) to (3.7,0.15);
\node at (5.2,-0.2) {.};
\end{tikzpicture}
\end{align*}
Similarly, using again Lemma \ref{lem:connline} (cf.\ Remark \ref{rem:local}),
the sum over $d_2\in \{d_3+1,\dots, n\}$ is
\begin{align*}
&\begin{tikzpicture}
\node at (-1,-1) {$- \displaystyle\sum\limits_{d_2=d_3+1}^{n}$};
\node at (0,-1) {\tiny{$k$}};
\node at (0.7,-1) {\tiny{$(\!k\!+\!1\!)$}};
\node at (1.4,-1) {.};
\node at (1.5,-1) {.};
\node at (1.6,-1) {.};
\node at (2.3,-1) {\tiny{$\underline{(\!k\!-\!1\!)}$}};
\node at (2.3,-1.5) {\tiny{$[d_{k-1}]$}};
\node at (3,-1) {.};
\node at (3.1,-1) {.};
\node at (3.2,-1) {.};
\node at (3.7,-1) {\tiny{$\underline{3}$}};
\node at (3.7,-1.5) {\tiny{$[d_{3}]$}};
\node at (4.2,-1) {.};
\node at (4.3,-1) {.};
\node at (4.4,-1) {.};
\node at (4.9,-0.5) {\tiny{$1$}};
\node at (4.9,-1) {\tiny{$\underline{2}$}};
\node at (4.9,-1.5) {\tiny{$[d_{2}]$}};
\node at (5.4,-1) {.};
\node at (5.5,-1) {.};
\node at (5.6,-1) {.};
\node at (6.1,-1) {\tiny{$n$}};
\draw[->] (0.1,-1) to (0.4,-1);
\draw[->] (1,-1) to (1.3,-1);
\draw[->] (1.7,-1) to (2,-1);
\draw[->] (2.6,-1) to (2.9,-1);
\draw[->] (3.3,-1) to (3.6,-1);
\draw[->] (3.8,-1) to (4.1,-1);
\draw[->] (4.5,-1) to (4.8,-1);
\draw[->] (5,-1) to (5.3,-1);
\draw[->] (5.7,-1) to (6,-1);
\draw[->] (4.9,-0.6) to (4.9,-0.85);
%
\end{tikzpicture} \notag \\
&\begin{tikzpicture}
\node at (-0.8,0) {$=(-1)^{2}$};
\node at (0,0) {\tiny{$k$}};
\node at (0.7,0) {\tiny{$(\!k\!+\!1\!)$}};
\node at (1.4,0) {.};
\node at (1.5,0) {.};
\node at (1.6,0) {.};
\node at (2.3,0) {\tiny{$\underline{(\!k\!-\!1\!)}$}};
\node at (2.3,-0.5) {\tiny{$[d_{k-1}]$}};
\node at (3,0) {.};
\node at (3.1,0) {.};
\node at (3.2,0) {.};
\node at (3.7,0) {\tiny{$\underline{3}$}};
\node at (3.7,-0.5) {\tiny{$[d_{3}]$}};
\node at (3.7,0.5) {\tiny{$2$}};
\node at (3.7,1) {\tiny{$1$}};
\node at (4.2,0) {.};
\node at (4.3,0) {.};
\node at (4.4,0) {.};
\node at (4.9,0) {\tiny{$n$}};
\draw[->] (0.1,0) to (0.4,0);
\draw[->] (1,0) to (1.3,0);
\draw[->] (1.7,0) to (2,0);
\draw[->] (2.6,0) to (2.9,0);
\draw[->] (3.3,0) to (3.6,0);
\draw[->] (3.8,0) to (4.1,0);
\draw[->] (4.5,0) to (4.8,0);
\draw[->] (3.7,0.4) to (3.7,0.15);
\draw[->] (3.7,0.9) to (3.7,0.6);
\node at (5.2,-0.2) {.};
\end{tikzpicture}
\end{align*}
Repeating the same argument $k$ times, we conclude that
\begin{align*} 
&\begin{tikzpicture}
\node at (-2.5,-0.1) {$\displaystyle\sum\limits_{\pi\in \mc M_n^k} \pi \La_n = (-1)^{k-2} \displaystyle\sum\limits_{d_{k-1}=2}^{n}$};
\node at (0,0) {\tiny{$k$}};
\node at (0.7,0) {\tiny{$(\!k\!+\!1\!)$}};
\node at (1.4,0) {.};
\node at (1.5,0) {.};
\node at (1.6,0) {.};
\node at (2.3,0) {\tiny{$\underline{(\!k\!-\!1\!)}$}};
\node at (2.3,-0.5) {\tiny{$[d_{k-1}]$}};
\node at (2.3,0.5) {\tiny{$(\!k\!-\!2\!)$}};
\node at (2.3, 1) {.};
\node at (2.3, 1.1) {.};
\node at (2.3, 1.2) {.};
\node at (2.3,1.7) {\tiny{$1$}};
\node at (3,0) {.};
\node at (3.1,0) {.};
\node at (3.2,0) {.};
\node at (3.7,0) {\tiny{$n$}};
\draw[->] (0.1,0) to (0.4,0);
\draw[->] (1,0) to (1.3,0);
\draw[->] (1.7,0) to (2,0);
\draw[->] (2.6,0) to (2.9,0);
\draw[->] (3.3,0) to (3.6,0);
\draw[->] (2.3,1.6) to (2.3,1.3);
\draw[->] (2.3,0.9) to (2.3,0.6);
\draw[->] (2.3,0.4) to (2.3,0.15);
\end{tikzpicture} \notag\\
&\;\;\;\;\;\;\;\;\;\;\;\;\;\;\;\;\;\;\;\;\;\begin{tikzpicture}
\node at (-1.2,0) {$=(-1)^{k-1}$};
\node at (0,0) {\tiny{$k$}};
\node at (0,0.5) {\tiny{$(\!k\!-\!1\!)$}};
\node at (0, 1) {.};
\node at (0, 1.1) {.};
\node at (0, 1.2) {.};
\node at (0,1.7) {\tiny{$1$}};
\node at (0.7,0) {\tiny{$(\!k\!+\!1\!)$}};
\node at (1.4,0) {.};
\node at (1.5,0) {.};
\node at (1.6,0) {.};
\node at (2.1,0) {\tiny{$n$}};
\draw[->] (0.1,0) to (0.4,0);
\draw[->] (1,0) to (1.3,0);
\draw[->] (1.7,0) to (2,0);
\draw[->] (0,1.6) to (0,1.3);
\draw[->] (0,0.9) to (0,0.6);
\draw[->] (0,0.4) to (0,0.15);
\end{tikzpicture} \notag \\
&\;\;\;\;\;\;\;\;\;\;\;\;\;\;\;\>\,\;\;\;\;\,= (-1)^{k-1} \La_n\,,
\end{align*}
proving the lemma.
\end{proof}

\subsection{Relation between the symmetry property and Harrison's conditions} \label{ssec:symmetry}

Recall that every $Y\in W^{n-1}_{\cl}(\Pi V)$ satisfies the symmetry property \eqref{eq:symmetry}. 

\begin{lemma} \label{lem:symmetry}
If\/ $Y\in W^{n-1}_{\cl}(\Pi V)$, then $Y^{\La_n}$ satisfies Harrison's relations \eqref{eq:harrisoncond}, 
hence it lies in the differential Harrison cohomology complex:
$$
Y^{\La_n}\in C^n_{\partial,\Har}(V)\,.
$$
Conversely, given $F\in C^n_{\partial,\Har}(V)$, there exists a unique $Y\in \gr^{n-1}W^{n-1}_{\cl}(\Pi V)$,
such that 
$$
Y^{\La_n}=F\,.
$$
Consequently, the linear map
\begin{equation*}
\gr^{n-1}W^{n-1}_{\cl}(\Pi V) \xrightarrow[]{\sim} C^{n}_{\partial,\Har}(V)
\,\,,\quad
Y\mapsto Y^{\La_n}
\end{equation*}
is bijective.
\end{lemma}
\begin{proof}
First, we prove that, if $Y\in W^{n-1}_{\cl}(\Pi V)$ satisfies 
the symmetry relations \eqref{eq:symmetry}, then $f=Y^{\La_n}$ satisfies 
Harrison's conditions \eqref{eq:harrisoncond}. By Lemma \ref{lem:identity} (cf.\  Remark \ref{rem:cycle}), 
we get
\begin{equation*} 
Y^{\La_n} = (-1)^{k-1} \sum_{\pi\in \mc M_n^k} Y^{\pi (\La_n)}\,.
\end{equation*} 
Evaluating both sides of this identity on $v_1\otimes \dots\otimes v_n \in V^{\otimes n}$, the left side is simply
$Y^{\La_n}(v_1\otimes \dots\otimes v_n)=f(v_1\otimes \dots\otimes v_n)$. 
On the right-hand side, we have
\begin{align*}
(-1)^{k-1} \sum_{\pi\in \mc M_n^k} & Y^{\pi (\La_n)} (v_1\otimes \dots\otimes v_n) 
= 
(-1)^{k-1} \sum_{\pi\in \mc M_n^k} (Y^{\pi^{-1}})^{\pi (\La_n)}(v_1\otimes \dots\otimes v_n) \\
& = (-1)^{k-1} \sum_{\pi\in \mc M_n^k} \sign(\pi) Y^{\La_n}(v_{\pi(1)}\otimes \dots\otimes v_{\pi(n)}) \\
& = L_k f (v_1\otimes \dots\otimes v_n)\,,
\end{align*}
by the definition \eqref{eq:lkf} of $L_k$. Hence, $f$ satisfies Harrison's conditions \eqref{eq:harrisoncond} 
as claimed.

We next turn to the second claim of the lemma. Let $F\in C^n_{\partial,\Har}(V)$, 
i.e., $F\colon V^{\otimes n}\rightarrow V$ is an $\mb F[\partial]$-module homomorphism 
satisfying Harrison's conditions \eqref{eq:harrisoncond}. 
Then the corresponding $Y\in \gr^{n-1}W^{n-1}_{\cl}(\Pi V)$, 
such that $Y^{\La_n}=F$, is defined as follows. For $\Ga\in R(n)$, or if $\Ga\in \mc L(n)$
is not connected, we set 
\begin{equation}\label{eq:s1}
Y^{\Ga}=0\,.
\end{equation}
For $\Ga\in \mc L(n)$ connected, there exists a unique $\tau\in S_n$ such that $\tau(1)=1$ 
and $\Ga=\tau (\La_n)$. We then set
\begin{equation}\label{eq:s2}
Y^{\Ga}(v_1\otimes\dots\otimes v_n)=\sign(\tau) F(v_{\tau(1)}\otimes \dots\otimes v_{\tau(n)})\,.
\end{equation}
In particular, for $\Ga=\La_n$, we have $\tau=1$ and $Y^{\La_n}=F$.

Since, by Theorem \ref{thm:lines}, $\mc L(n)$ is a basis for the vector space $\mb F\mc G(n)/R(n)$,
equations \eqref{eq:s1} and \eqref{eq:s2} determine a unique well-defined element 
$Y \in \gr^{n-1} \mc P_{\cl}(\Pi V)(n)$. 
Indeed, $Y$ satisfies the cycle relations \eqref{eq:cycle1} and \eqref{eq:cycle2},
by \eqref{eq:s1}, Theorem \ref{thm:lines} and Remark \ref{rem:cycle}.
The first sesquilinearity condition \eqref{eq:sesq1} is obvious, and the second condition \eqref{eq:sesq2} follows from
\eqref{eq:sesq2.1}.

To prove that $Y\in \gr^{n-1} W^{n-1}_{\cl}(\Pi V)$, it remains to show that $Y$ satisfies 
the symmetry conditions \eqref{eq:symmetry}.
Obviously, the action of the symmetric group $S_n$ preserves $R(n)$ and the set of non-connected lines. 
Hence, when we evaluate \eqref{eq:symmetry} on $\Ga\in R(n)$ or on a non-connected $n$-line 
$\Ga\in \mc L(n)$, we get $0=0$.

We are left to prove that \eqref{eq:symmetry} holds when evaluated on a connected line $\Ga\in \mc L(n)$, 
which, as remarked above, can be obtained as $\Ga=\tau (\La_n)$, for a unique $\tau\in S_n$ 
such that $\tau(1)=1$. The right-hand side of \eqref{eq:symmetry}, when evaluated on such a $\Ga$ 
is given by \eqref{eq:s2}. The left-hand side is, by \eqref{eq:actclassicoperad}, 
\begin{align}\label{eq:s3}
(Y^{\sigma})^{\Ga} (v_1\otimes\dots\otimes v_n)&= Y^{\sigma\tau (\La_n)} 
(\sigma(v_1\otimes\dots\otimes v_n)) \notag \\
&= \sign(\sigma)Y^{\sigma\tau (\La_n)} (v_{\sigma(1)}\otimes\dots\otimes v_{\sigma_{n}})\,.
\end{align}
By Lemma \ref{lem:identity}, we have, modulo $R(n)$, 
\begin{equation*}
\sigma\tau(\La_n) \equiv (-1)^{\tau^{-1}\sigma^{-1}(1)-1}
\sum_{\pi \in \mc M_{n}^{\tau^{-1}\sigma^{-1}(1)}} \sigma\tau \pi (\La_n)\,.
\end{equation*}
Hence, by Remark \ref{rem:cycle}, the right-hand side of \eqref{eq:s3} becomes
\begin{equation}\label{eq:s4}
\sign(\sigma)(-1)^{\tau^{-1}\sigma^{-1}(1)-1}\sum_{\pi \in \mc M_{n}^{\tau^{-1}\sigma^{-1}(1)}} 
Y^{\sigma\tau \pi (\La_n)}(v_{\sigma^{-1}(1)}\otimes\dots\otimes v_{\sigma^{-1}(n)})\,.
\end{equation}
Note that, if $\pi\in \mc M_{n}^{\tau^{-1}\sigma^{-1}(1)}$, then $\sigma\tau\pi(1)=1$. 
Hence, we can apply formula \eqref{eq:s2} to get
\begin{equation}\label{eq:s5}
Y^{\sigma\tau\pi(\La_n)} (v_{\sigma^{-1}(1)}\otimes\dots\otimes v_{\sigma^{-1}(n)})
=
\sign (\sigma\tau\pi) F(v_{\tau\pi(1)}\otimes\dots\otimes v_{\tau\pi(n)})\,.
\end{equation}
Combining \eqref{eq:s3}--\eqref{eq:s5}, we get, by \eqref{eq:drop} and the definition \eqref{eq:lkf} of $L_k$, 
\begin{align*}
& (Y^{\sigma})^{\Ga}(v_1\otimes\dots\otimes v_n) \\
&= \sign(\tau)(-1)^{\tau^{-1}\sigma^{-1}(1)-1} \sum_{\pi\in \mc M_{n}^{\tau^{-1}\sigma^{-1}(1)}} 
\sign(\pi) F (v_{\tau\pi(1)}\otimes\dots\otimes v_{\tau\pi(n)}) \\
&=\sign(\tau)(L_{\tau^{-1}\sigma^{-1}(1)}F)(v_{\tau(1)}\otimes\dots\otimes v_{\tau}(n))\,,
\end{align*}
which equals \eqref{eq:s2} by Harrison's conditions \eqref{eq:harrisoncond}. 

Hence, $Y$ is a well-defined element of $\gr^{n-1}W^{n-1}_{\cl}(\Pi V)$, such that $Y^{\La_n}=F$, as required. 
The uniqueness of such a $Y$ is obvious since, by Theorem \ref{thm:lines}, 
$Y\in \gr^{n-1}W^{n-1}_{\cl}(\Pi V)$ is uniquely determined by its value on $\La_n$. 
\end{proof}

\subsection{Relation between $\ad X$ and the Hochschild differential} \label{ssec:differentials}

\begin{lemma} \label{lem:differential}
For\/ $Y\in W^{n-1}_\cl (\Pi V)$ and\/ $X\in W^1_\cl(\Pi V)$ defined in \eqref{eq:PVAstructure}, we have
$$
[X,Y]^{\La_{n+1}}(v_1\otimes \dots\otimes v_{n+1})=(-1)^{n+1} d(Y^{\La_n})(v_1\otimes\dots\otimes v_{n+1})
\,,
$$
where $\La_n$ is as in \eqref{eq:Ln} and $d$ is the Hochschild differential \eqref{eq:hochdifferential}.
\end{lemma}
\begin{proof}
Recall that, by definition \eqref{eq:liebracket}, the adjoint action of $X$ on $Y$ is given by:
\begin{equation} \label{eq:nliebracket}
\begin{split}
[X,Y]&=X \Box Y -(-1)^{n-1}Y\Box X \\
&= \sum_{\sigma \in S_{n,1}}(X\circ_1 Y)^{\sigma^{-1}} 
+ (-1)^n \sum_{\tau \in S_{2,n-1}}(Y\circ_1 X)^{\tau^{-1}} \,,
\end{split}
\end{equation}	
since $\bar{p}(X)=1$ and $\bar{p}(Y)=n-1$. The elements of $S_{n,1}$ are 
\begin{equation} \label{eq:Sn1}
\begin{split}
\sigma_k &= \left(\begin{array}{cccccccc|c}
1 & 2 & \cdots & k-1 & k & k+1 & \cdots & n  & n+1 \\
1 & 2 & \cdots & k-1 & k+1 & k+2 & \cdots & n+1 & k
\end{array} \right) \\
&= (k \;\; k+1\;\; \cdots\; \; n+1)\,,
\end{split}
\end{equation}
for $1\le k\le n+1$.
The elements of $S_{2,n-1}$ are 
\begin{equation} \label{eq:S2n-1a}
\tau_{i,j} = \left(\begin{array}{cc|cccccccccc}
1 & 2  & 3 & \cdots & i+1 & i+2 & \cdots & j   & j+1 & \cdots & n+1 \\
i & j & 1 & \cdots & i-1 & i+1 & \cdots & j-1 & j+1 & \cdots & n+1
\end{array} \right) \,,
\end{equation}
for $1\leq i < j-1 \leq n$, and 
\begin{equation} \label{eq:S2n-1b}
\tau_{i,i+1} = \left(\begin{array}{cc|cccccccccc}
1 & 2  & 3 & \cdots & i+1 & i+2 & \cdots & n+1 \\
i & i+1 & 1 & \cdots & i-1 & i+2 & \cdots & n+1
\end{array} \right) \,,
\end{equation}
for $1\leq i \leq n$.
%

Using \eqref{eq:Sn1}, we evaluate 
\begin{equation} \label{eq:prodsigmak}
\begin{split}
\bigl(&(X\circ_1 Y)^{\sigma_k^{-1}}\bigr)^{\La_{n+1}}(v_1 \otimes \cdots \otimes v_{n+1}) \\			
&=(-1)^{n+1-k}(X\circ_1 Y)^{\sigma_k^{-1}(\La_{n+1})}
(v_1\otimes\cdots\otimes v_{k-1}\otimes v_{k+1}\otimes \cdots\otimes v_{n+1}\otimes v_k)\,.	
\end{split}
\end{equation}
For $k=1$, by the symmetric group's action described in Section \ref{ssec:ngraphs}, we have
\begin{align*}
&\begin{tikzpicture}[scale=1.5]
\node at (-1,0) {$\sigma_1^{-1}(\La_{n+1})= $};
\draw[fill] (0,0) circle [radius=0.04];
\node at (0,-0.2){\tiny{$\sigma_{1}^{-1}(1)\;\;$}};
\draw[fill] (0.5,0) circle [radius=0.04];
\node at (0.5,-0.2){\tiny{$\;\;\sigma_{1}^{-1}(2)$}};
\node at (1,0) {.};
\node at (1.1,0) {.};
\node at (1.2,0) {.};
\draw[fill] (1.7,0) circle [radius=0.04];
\node at (1.7,-0.2){\tiny{$\sigma_{1}^{-1}(n\!+\!1)$}};
\draw[->] (0.1,0) to (0.4,0);
\draw[->] (0.6,0) to (0.9,0);
\draw[->] (1.3,0) to (1.6,0);
\end{tikzpicture} \\
&\qquad\qquad\quad
\begin{tikzpicture}[scale=1.5]
\node at (-0.2,0) {$=$};
\draw[fill] (0,0) circle [radius=0.04];
\node at (0,-0.2){\tiny{$1$}};
\node at (0.5,0) {.};
\node at (0.7,0) {.};
\node at (0.6,0) {.};
\draw[fill] (1.2,0) circle [radius=0.04];
\node at (1.2,-0.2){\tiny{$n$}};
\draw[fill] (1.7,0) circle [radius=0.04];
\node at (1.7,-0.2){\tiny{$n\!+\!1$}};
\draw[->] (0.1,0) to (0.4,0);
\draw[->] (0.8,0) to (1.1,0);
\draw[<-] (0,0.1) to [out=90,in=90] (1.7,0.1);
\node at (2,0) {.};
\end{tikzpicture} 
\end{align*}
When we apply the cocomposition map $\Delta^{n,1}$ to this graph, we get
\begin{equation*}
\begin{tikzpicture}[scale=1.5]
\node at (-1.4,0) {$\Delta_0^{n,1}(\sigma_1^{-1}(\La_{n+1}))\;=$};
\draw[fill] (0,0) circle [radius=0.04];
\node at (0,-0.2){\tiny{$1$}};
\draw[fill] (0.5,0) circle [radius=0.04];
\node at (0.5,-0.2){\tiny{$2$}};
\draw[<-] (0.1,0) to (0.4,0);
\end{tikzpicture} 
\end{equation*}
and
\begin{equation*}
\begin{tikzpicture}[scale=1.5]
\node at (-1.4,0) {$\Delta_1^{n,1}(\sigma_1^{-1}(\La_{n+1}))\;=$};
\draw[fill] (0,0) circle [radius=0.04];
\node at (0,-0.2){\tiny{$1$}};
\draw[fill] (0.5,0) circle [radius=0.04];
\node at (0.5,-0.2){\tiny{$2$}};
\node at (1,0) {.};
\node at (1.1,0) {.};
\node at (1.2,0) {.};
\draw[fill] (1.7,0) circle [radius=0.04];
\node at (1.7,-0.2){\tiny{$n$}};
\draw[->] (0.1,0) to (0.4,0);
\draw[->] (0.6,0) to (0.9,0);
\draw[->] (1.3,0) to (1.6,0);
\node at (2.2,0) {$\;=\;\La_n\;.$};
\end{tikzpicture} 
\end{equation*}
Hence, by the definition \eqref{eq:composition} of the composition map, \eqref{eq:prodsigmak} becomes
\begin{equation} \label{eq:m0} 
\begin{split}
\bigl((X\circ_1 &Y)^{\sigma_1^{-1}}\bigr)^{\La_{n+1}}(v_1 \otimes \dots \otimes v_{n+1}) \\
&= (-1)^{n}(X\circ_1 Y)^{\sigma_1^{-1}(\La_{n+1})} (v_2\otimes\dots\otimes v_{n+1}\otimes v_1)  \\
&= (-1)^n X^{\bullet\!\leftarrow\!\!\!\!-\!\bullet}(Y^{\La_n}(v_2\otimes \dots\otimes v_{n+1})\otimes v_1)  \\
&= (-1)^{n+1} Y^{\La_n}(v_2\otimes\dots\otimes v_{n+1})\, v_1 \,.
\end{split}
\end{equation}
Similarly, for $k=n+1$, we have $\sigma_{n+1}=1$, and applying the cocomposition map 
$\Delta^{n,1}$ to $\sigma_{n+1}^{-1}(\La_{n+1})=\La_{n+1}$, we get:
\begin{equation*}
\begin{tikzpicture}[scale=1.5]
\node at (-1.4,0) {$\Delta_0^{n,1}(\sigma_{n+1}^{-1}(\La_{n+1}))\;=$};
\draw[fill] (0,0) circle [radius=0.04];
\node at (0,-0.2){\tiny{$1$}};
\draw[fill] (0.5,0) circle [radius=0.04];
\node at (0.5,-0.2){\tiny{$2$}};
\draw[->] (0.1,0) to (0.4,0);
\end{tikzpicture} 
\end{equation*}
and
\begin{equation*}
\begin{tikzpicture}[scale=1.5]
\node at (-1.4,0) {$\Delta_1^{n,1}(\sigma_{n+1}^{-1}(\La_{n+1}))\;=$};
\draw[fill] (0,0) circle [radius=0.04];
\node at (0,-0.2){\tiny{$1$}};
\draw[fill] (0.5,0) circle [radius=0.04];
\node at (0.5,-0.2){\tiny{$2$}};
\node at (1,0) {.};
\node at (1.1,0) {.};
\node at (1.2,0) {.};
\draw[fill] (1.7,0) circle [radius=0.04];
\node at (1.7,-0.2){\tiny{$n$}};
\draw[->] (0.1,0) to (0.4,0);
\draw[->] (0.6,0) to (0.9,0);
\draw[->] (1.3,0) to (1.6,0);
\node at (2.2,0) {$\;=\;\La_n\;.$};
\end{tikzpicture} 
\end{equation*}
Then \eqref{eq:prodsigmak} becomes
\begin{equation} \label{eq:m1} 
\begin{split}
\bigl((X\circ_1 &Y)^{\sigma_{n+1}^{-1}}\bigr)^{\La_{n+1}}(v_1 \otimes \dots \otimes v_{n+1}) \\
&= X^{\bullet\!\to\!\!\!\!-\!\bullet}(Y^{\La_n}(v_1\otimes \dots\otimes v_{n})\otimes v_{n+1})  \\
&= Y^{\La_n}(v_1\otimes\dots\otimes v_{n})\, v_{n+1}\,.
\end{split}
\end{equation}
Furthermore, for $2\leq k \leq n$, we have
\begin{align*}
&\begin{tikzpicture}[scale=1.5]
\node at (-1,0) {$\sigma_k^{-1}(\La_{n+1})= $};
\draw[fill] (0,0) circle [radius=0.04];
\node at (0,-0.2){\tiny{$\sigma_{k}^{-1}(1)\;\;$}};
\draw[fill] (0.5,0) circle [radius=0.04];
\node at (0.5,-0.2){\tiny{$\;\;\sigma_{k}^{-1}(2)$}};
\node at (1,0) {.};
\node at (1.1,0) {.};
\node at (1.2,0) {.};
\draw[fill] (1.7,0) circle [radius=0.04];
\node at (1.7,-0.2){\tiny{$\sigma_{k}^{-1}(n\!+\!1)$}};
\draw[->] (0.1,0) to (0.4,0);
\draw[->] (0.6,0) to (0.9,0);
\draw[->] (1.3,0) to (1.6,0);
\end{tikzpicture} \\
&\qquad\qquad\quad
\begin{tikzpicture}[scale=1.5]
\node at (-0.2,0) {$=$};
\draw[fill] (0,0) circle [radius=0.04];
\node at (0,-0.2){\tiny{$1$}};
\node at (0.5,0) {.};
\node at (0.7,0) {.};
\node at (0.6,0) {.};
\draw[fill] (1.2,0) circle [radius=0.04];
\node at (1.2,-0.2){\tiny{$k\!-\!1$}};
\draw[fill] (1.7,0) circle [radius=0.04];
\node at (1.7,-0.2){\tiny{$k$}};
\node at (2.2,0) {.};
\node at (2.3,0) {.};
\node at (2.4,0) {.};
\draw[fill] (2.9,0) circle [radius=0.04];
\node at (2.9,-0.2){\tiny{$n$}};
\draw[fill] (3.4,0) circle [radius=0.04];
\node at (3.4,-0.2){\tiny{$n\!+\!1$}};
\draw[->] (0.1,0) to (0.4,0);
\draw[->] (0.8,0) to (1.1,0);
\draw[->] (1.8,0) to (2.1,0);
\draw[->] (2.5,0) to (2.8,0);
\draw[->] (1.2,0.1) to [out=90,in=90] (3.4,0.1);
\draw[<-] (1.7,-0.1) to [out=-90,in=-90] (3.4,-0.1);
\node at (4,0) {.};
\end{tikzpicture} 
\end{align*}
Hence, applying the cocomposition map $\Delta^{n,1}$ we get
\begin{equation*}
\begin{tikzpicture}[scale=1.5]
\node at (-1.4,0) {$\Delta_0^{n,1}(\sigma_{k}^{-1}(\La_{n+1}))\;=$};
\draw[fill] (0,0) circle [radius=0.04];
\node at (0,-0.2){\tiny{$1$}};
\draw[fill] (0.5,0) circle [radius=0.04];
\node at (0.5,-0.2){\tiny{$2$}};
\draw[->] (0,0.1) to [out=90,in=90] (0.5,0.1);
\draw[<-] (0,-0.1) to [out=-90,in=-90] (0.5,-0.1);
\end{tikzpicture} 
\end{equation*}
which has a cycle. Therefore,
\begin{equation}\label{eq:m2}
\bigl((X\circ_1 Y)^{\sigma_{k}^{-1}}\bigr)^{\La_{n+1}} = X^{\begin{tikzpicture}
\draw[fill] (0,0) circle [radius=0.04];
\draw[fill] (0.5,0) circle [radius=0.04];
\draw[->] (0,0.1) to [out=90,in=90] (0.5,0.1);
\draw[<-] (0,-0.1) to [out=-90,in=-90] (0.5,-0.1);
\end{tikzpicture} }(\cdots) =0\,.
\end{equation}

Next, we compute
$
\bigl((Y\circ_1 X)^{\tau_{i,j}^{-1}}\bigr)^{\La_{n+1}}(v_1 \otimes \dots \otimes v_{n+1})
$
%
where $\tau_{i,j}$ is defined in \eqref{eq:S2n-1a}. By the symmetric group's action described in Section \ref{ssec:ngraphs} we have
\begin{align*}
&\begin{tikzpicture}[scale=1.5]
\node at (-1,0) {$\tau_{i,j}^{-1}(\La_{n+1})= $};
\draw[fill] (0,0) circle [radius=0.04];
\node at (0,-0.2){\tiny{$\tau_{i,j}^{-1}(1)\;\;$}};
\draw[fill] (0.5,0) circle [radius=0.04];
\node at (0.5,-0.2){\tiny{$\;\;\tau_{i,j}^{-1}(2)$}};
\node at (1,0) {.};
\node at (1.1,0) {.};
\node at (1.2,0) {.};
\draw[fill] (1.7,0) circle [radius=0.04];
\node at (1.7,-0.2){\tiny{$\tau_{i,j}^{-1}(n\!+\!1)$}};
\draw[->] (0.1,0) to (0.4,0);
\draw[->] (0.6,0) to (0.9,0);
\draw[->] (1.3,0) to (1.6,0);
\end{tikzpicture} \\
&\qquad\qquad\quad
\begin{tikzpicture}[scale=1.5]
\node at (-0.3,0) {$=$};
\draw[fill] (0,0) circle [radius=0.04];
\node at (0,-0.2){\tiny{$1$}};
\draw[fill] (0.5,0) circle [radius=0.04];
\node at (0.5,-0.2){\tiny{$2$}};
\draw[fill] (1,0) circle [radius=0.04];
\node at (1,-0.2){\tiny{$3$}};
\node at (1.5,0) {.};
\node at (1.6,0) {.};
\node at (1.7,0) {.};
\draw[fill] (2.2,0) circle [radius=0.04];
\node at (2.2,-0.2){\tiny{$i\!+\!1$}};
\draw[fill] (2.7,0) circle [radius=0.04];
\node at (2.7,-0.2){\tiny{$i\!+\!2$}};
\node at (3.2,0) {.};
\node at (3.3,0) {.};
\node at (3.4,0) {.};
\draw[fill] (3.9,0) circle [radius=0.04];
\node at (3.9,-0.2){\tiny{$j$}};
\draw[fill] (4.4,0) circle [radius=0.04];
\node at (4.4,-0.2){\tiny{$j\!+\!1$}};
\node at (4.9,0) {.};
\node at (5,0) {.};
\node at (5.1,0) {.};
\draw[fill] (5.6,0) circle [radius=0.04];
\node at (5.6,-0.2){\tiny{$n\!+\!1$}};
\draw[->] (1.1,0) to (1.4,0);
\draw[->] (1.8,0) to (2.1,0);
\draw[->] (2.8,0) to (3.1,0);
\draw[->] (3.5,0) to (3.8,0);
\draw[->] (4.5,0) to (4.8,0);
\draw[->] (5.2,0) to (5.5,0);
\draw[->] (0,0.1) to [out=90,in=90] (2.7,0.1);
\draw[<-] (0,0.1) to [out=45,in=135] (2.2,0.1);
\draw[->] (0.5,-0.1) to [out=-90,in=-90] (4.4,-0.1);
\draw[<-] (0.5,-0.1) to [out=-45,in=-135] (3.9,-0.1);
\node at (6,0) {.};
\end{tikzpicture} 
\end{align*}
Hence, applying the cocomposition $\Delta^{2,1,\dots,1}$ we get 
\begin{equation*}
\begin{tikzpicture}[scale=1.5]
\node at (-1,0) {$\Delta^{2,1,\dots,1}_0(\tau_{i,j}^{-1}(\La_{n+1}))=$};
\draw[fill] (0.5,0) circle [radius=0.04];
\node at (0.5,-0.2){\tiny{$1$}};
\draw[fill] (1,0) circle [radius=0.04];
\node at (1,-0.2){\tiny{$2$}};
\node at (1.5,0) {.};
\node at (1.6,0) {.};
\node at (1.7,0) {.};
\draw[fill] (2.2,0) circle [radius=0.04];
\node at (2.2,-0.2){\tiny{$i$}};
\draw[fill] (2.7,0) circle [radius=0.04];
\node at (2.7,-0.2){\tiny{$i\!+\!1$}};
\node at (3.2,0) {.};
\node at (3.3,0) {.};
\node at (3.4,0) {.};
\draw[fill] (3.9,0) circle [radius=0.04];
\node at (3.9,-0.2){\tiny{$j\!-\!1$}};
\draw[fill] (4.4,0) circle [radius=0.04];
\node at (4.4,-0.2){\tiny{$j$}};
\node at (4.9,0) {.};
\node at (5,0) {.};
\node at (5.1,0) {.};
\draw[fill] (5.6,0) circle [radius=0.04];
\node at (5.6,-0.2){\tiny{$n$}};
\draw[->] (1.1,0) to (1.4,0);
\draw[->] (1.8,0) to (2.1,0);
\draw[->] (2.8,0) to (3.1,0);
\draw[->] (3.5,0) to (3.8,0);
\draw[->] (4.5,0) to (4.8,0);
\draw[->] (5.2,0) to (5.5,0);
\draw[->] (0.5,0.1) to [out=90,in=90] (2.7,0.1);
\draw[<-] (0.5,0.1) to [out=45,in=135] (2.2,0.1);
\draw[->] (0.5,-0.1) to [out=-90,in=-90] (4.4,-0.1);
\draw[<-] (0.5,-0.1) to [out=-45,in=-135] (3.9,-0.1);
\node at (6,0) {,};
\end{tikzpicture} 
\end{equation*}
which has a cycle. Therefore,
\begin{equation}\label{eq:m3}
\bigl((Y\circ_1 X)^{\tau_{i,j}^{-1}}\bigr)^{\La_{n+1}} 
= 
Y^{\Delta_0^{2,1,\dots,1}(\tau_{i,j}^{-1}(\La_{n+1}))}(\cdots) =0\,.
\end{equation}

Finally, we evaluate 
$
\bigl((Y\circ_1 X)^{\tau_{i,i+1}^{-1}}\bigr)^{\La_{n+1}}(v_1 \otimes \dots \otimes v_{n+1})
$
%
with $\tau_{i,i+1}$ defined in \eqref{eq:S2n-1b}. In this case, we have
\begin{align*}
&\begin{tikzpicture}[scale=1.5]
\node at (-1,0) {$\tau_{i,i+1}^{-1}(\La_{n+1})= $};
\draw[fill] (0,0) circle [radius=0.04];
\node at (0,-0.2){\tiny{$\tau_{i,i+1}^{-1}(1)\;\;\;\;$}};
\draw[fill] (0.5,0) circle [radius=0.04];
\node at (0.5,-0.2){\tiny{$\;\;\;\;\tau_{i,i+1}^{-1}(2)$}};
\node at (1,0) {.};
\node at (1.1,0) {.};
\node at (1.2,0) {.};
\draw[fill] (1.7,0) circle [radius=0.04];
\node at (1.7,-0.2){\tiny{$\tau_{i,i+1}^{-1}(n\!+\!1)$}};
\draw[->] (0.1,0) to (0.4,0);
\draw[->] (0.6,0) to (0.9,0);
\draw[->] (1.3,0) to (1.6,0);
\end{tikzpicture} \\
&\qquad\qquad\quad\;\;
\begin{tikzpicture}[scale=1.5]
\node at (-0.3,0) {$=$};
\draw[fill] (0,0) circle [radius=0.04];
\node at (0,-0.2){\tiny{$1$}};
\draw[fill] (0.5,0) circle [radius=0.04];
\node at (0.5,-0.2){\tiny{$2$}};
\draw[fill] (1,0) circle [radius=0.04];
\node at (1,-0.2){\tiny{$3$}};
\node at (1.5,0) {.};
\node at (1.6,0) {.};
\node at (1.7,0) {.};
\draw[fill] (2.2,0) circle [radius=0.04];
\node at (2.2,-0.2){\tiny{$i\!+\!1$}};
\draw[fill] (2.7,0) circle [radius=0.04];
\node at (2.7,-0.2){\tiny{$i\!+\!2$}};
\node at (3.2,0) {.};
\node at (3.3,0) {.};
\node at (3.4,0) {.};
\draw[fill] (3.9,0) circle [radius=0.04];
\node at (3.9,-0.2){\tiny{$n\!+\!1$}};
\draw[->] (0.1,0) to (0.4,0);
\draw[->] (1.1,0) to (1.4,0);
\draw[->] (1.8,0) to (2.1,0);
\draw[->] (2.8,0) to (3.1,0);
\draw[->] (3.5,0) to (3.8,0);
\draw[<-] (0,0.1) to [out=90,in=90] (2.2,0.1);
\draw[->] (0.5,-0.1) to [out=-90,in=-90] (2.7,-0.1);
\node at (4.5,0) {.};
\end{tikzpicture} 
\end{align*}
Hence, applying the cocomposition $\Delta^{2,1,\dots,1}$ we get 
\begin{equation*}
\begin{tikzpicture}[scale=1.5]
\node at (-1,0) {$\Delta^{2,1,\dots,1}_0(\tau_{i,i+1}^{-1}(\La_{n+1}))=$};
\draw[fill] (0.5,0) circle [radius=0.04];
\node at (0.5,-0.2){\tiny{$1$}};
\draw[fill] (1,0) circle [radius=0.04];
\node at (1,-0.2){\tiny{$2$}};
\node at (1.5,0) {.};
\node at (1.6,0) {.};
\node at (1.7,0) {.};
\draw[fill] (2.2,0) circle [radius=0.04];
\node at (2.2,-0.2){\tiny{$i$}};
\draw[fill] (2.7,0) circle [radius=0.04];
\node at (2.7,-0.2){\tiny{$i\!+\!1$}};
\node at (3.2,0) {.};
\node at (3.3,0) {.};
\node at (3.4,0) {.};
\draw[fill] (3.9,0) circle [radius=0.04];
\node at (3.9,-0.2){\tiny{$n$}};
\draw[->] (1.1,0) to (1.4,0);
\draw[->] (1.8,0) to (2.1,0);
\draw[->] (2.8,0) to (3.1,0);
\draw[->] (3.5,0) to (3.8,0);
\draw[<-] (0.5,0.1) to [out=90,in=90] (2.2,0.1);
\draw[->] (0.5,-0.1) to [out=-90,in=-90] (2.7,-0.1);
%
\end{tikzpicture} 
\end{equation*}
and
\begin{equation*}
\begin{tikzpicture}[scale=1.5]
\node at (-1.6,0) {$\Delta_1^{2,1,\dots,1}(\tau_{i,i+1}^{-1}(\La_{n+1}))\;=$};
\draw[fill] (0,0) circle [radius=0.04];
\node at (0,-0.2){\tiny{$1$}};
\draw[fill] (0.5,0) circle [radius=0.04];
\node at (0.5,-0.2){\tiny{$2$}};
\draw[->] (0.1,0) to (0.4,0);
\node at (1,0) {.};
\end{tikzpicture} 
\end{equation*}
Therefore, by definition \eqref{eq:composition} of the composition map, we obtain
\begin{align}\label{eq:m4}
&\bigl((Y\circ_1 X)^{\tau_{i,j}^{-1}}\bigr)^{\La_{n+1}} (v_1\otimes\dots\otimes v_n)= \notag \\
&=(Y\circ_1 X)^{\tau_{i,j}^{-1}(\La_{n+1})}(v_i\otimes v_{i+1}\otimes v_1\otimes \dots\otimes v_{i-1}\otimes v_{i+2}\otimes \dots\otimes v_{n+1}) \notag \\
&=Y^{\Delta^{2,1,\dots,1}_0(\tau_{i,i+1}^{-1}(\La_{n+1}))} \bigl(X^{\bullet\!\to\!\!\!\!-\!\bullet}(v_i\otimes v_{i+1})\otimes v_1\otimes \dots\otimes v_{i-1}\otimes v_{i+2}\otimes \dots\otimes v_{n+1}\bigr) \notag\\
&=Y^{\Delta^{2,1,\dots,1}_0(\tau_{i,i+1}^{-1}(\La_{n+1}))}(v_i v_{i+1}\otimes v_1\otimes \dots\otimes v_{i-1}\otimes v_{i+2}\otimes \dots\otimes v_{n+1}) \,.
\end{align}
Note that
\begin{equation*}
\Delta^{2,1,\dots,1}_0(\tau_{i,i+1}^{-1}(\La_{n+1}))=\sigma(\La_n)\,,
\end{equation*}
where $\sigma=(1\;2\,\cdots\,i)\in S_n$ is the $i$-cycle. Hence, by the symmetry property \eqref{eq:symmetry}, 
we can replace $Y$ by $Y^{{\sigma}^{-1}}$ in the right-hand side of \eqref{eq:m4} to get
\begin{equation}\label{eq:m5}
(-1)^{i+1} Y^{\La_n}(v_1\otimes \dots\otimes v_{i-1}\otimes v_i v_{i+1}\otimes v_{i+2}\otimes 
\dots\otimes v_{n+1})\,.
\end{equation}
Combining equations \eqref{eq:nliebracket} and \eqref{eq:m0}--\eqref{eq:m5}, we conclude that
\begin{align*}
[X,&Y]^{\La_{n+1}} (v_1\otimes\dots\otimes v_n) \\
&=(-1)^{n+1} v_1 \, Y^{\La_n}(v_2\otimes\dots\otimes v_{n+1})
+ Y^{\La_n}(v_1\otimes\dots\otimes v_{n})\, v_{n+1} \\
&\;\;\;+ (-1)^n \sum_{i=1}^{n} (-1)^{i+1} Y^{\La_n}(v_1\otimes \dots\otimes v_{i-1}\otimes v_i v_{i+1}
\otimes v_{i+2}\otimes \dots\otimes v_{n+1}) \\
&=(-1)^{n+1} d(Y^{\La_n})(v_1\otimes\dots\otimes v_n)\,,
\end{align*}
completing the proof.
\end{proof}
\begin{corollary}\label{cor:differential}
Let\/ $X$ be defined in \eqref{eq:PVAstructure} and\/ $Y\in W^{n-1}_\cl (\Pi V)$ be such that $[X,Y]=0$. 
Then\/ $Y^{\La_n}$ is a cocyle in the Hochschild cohomology.
\end{corollary}
\begin{proof}
Obvious, by Lemma \ref{lem:differential}.
\end{proof}

\subsection{Proof of Theorem \ref{thm:main}} \label{ssec:proof}

By Lemma \ref{lem:symmetry}, given $Y\in W^{n-1}_{\cl}(\Pi V)$, $Y^{\La_n}$ is a cochain 
in the differential Harrison complex. Conversely, for any $F\in C^{n}_{\partial,\Har}(V)$, 
there is a unique $Y\in \gr^{n-1}W^{n-1}_{\cl}(\Pi V)$ such that $F=Y^{\La_n}$. 
Hence, the following diagram is well defined:
\begin{equation}\label{eq:diagram}
\begin{array}{rccc}
Y\ni & W^{n-1}_{\cl}(\Pi V) & \xrightarrow[]{\ad X} & W^n_{\cl}(\Pi V) \\
\downarrow\;\;\;\; & \downarrow & & \downarrow \\
Y^{\La_n}\ni & C^{n}_{\partial,\Har}(V) & \xrightarrow[]{\;\; d \;\;} & C^{n+1}_{\partial,\Har}(V)
\end{array}\;\;\;\,,
\end{equation}
where the vertical maps are surjective and restrict to bijective maps on top degree: 
$\gr^{n-1}W^{n-1}_{\cl}(\Pi V)\xrightarrow[]{\sim} C^n_{\partial,\Har}(V)$. 
Lemma \ref{lem:differential} says that, up to a sign, the diagram \eqref{eq:diagram} is commutative. 
This completes the proof of the theorem.



\end{document}